\documentclass[a4paper]{amsart}  
\usepackage{lscape,exscale, fullpage,dsfont,lscape,pdflscape}
\usepackage{amssymb,amsmath,amsthm,geometry,graphicx} 
\usepackage{stmaryrd}
\usepackage{amsfonts,amscd}
\usepackage{psfrag,xspace} 
\usepackage[latin1]{inputenc}
\usepackage[T1]{fontenc}
\usepackage[active]{srcltx}
\usepackage{bbm}
\usepackage[mathscr]{euscript}
\usepackage{tikz}
\tikzset{math3d/.style={x= {(1cm,0cm)}, z={(0cm,1cm)},y={(0.353cm,0.353cm)}}}
\usepackage{psfrag} 
\usepackage{fullpage,url}
\usepackage{latexsym} 
\usepackage{epsfig} 
\usepackage[all]{xy}
\usepackage{mathrsfs}
\usepackage{verbatim,epic,eepic}

\newcommand{\lra}{\longrightarrow}
\newcommand{\ra}{\rightarrow}

\newcommand{\mbc}{\mathbb{C}}

\newcommand{\mbk}{\mathbb{K}}
\newcommand{\mbl}{\mathbb{L}}

\newcommand{\mbn}{\mathbb{N}}

\newcommand{\mbp}{\mathbb{P}}
\newcommand{\mbq}{\mathbb{Q}}
\newcommand{\mbr}{\mathbb{R}}

\newcommand{\mbz}{\mathbb{Z}}

\newcommand{\mca}{\mathcal{A}}

\newcommand{\mcc}{\mathcal{C}}
\newcommand{\mcd}{\mathcal{D}}
\newcommand{\mce}{\mathcal{E}}
\newcommand{\mcf}{\mathcal{F}}
\newcommand{\mcg}{\mathcal{G}}
\newcommand{\mch}{\mathcal{H}}
\newcommand{\mci}{\mathcal{I}}
\newcommand{\mcj}{\mathcal{J}}
\newcommand{\mck}{\mathcal{K}}
\newcommand{\mcl}{\mathcal{L}}
\newcommand{\mcm}{\mathcal{M}}
\newcommand{\mcn}{\mathcal{N}}
\newcommand{\mco}{\mathcal{O}}
\newcommand{\mcp}{\mathcal{P}}
\newcommand{\mcq}{\mathcal{Q}}
\newcommand{\mcr}{\mathcal{R}}
\newcommand{\mcs}{\mathcal{S}}
\newcommand{\mct}{\mathcal{T}}
\newcommand{\mcu}{\mathcal{U}}
\newcommand{\mcv}{\mathcal{V}}

\newcommand{\mcx}{\mathcal{X}}
\newcommand{\mcy}{\mathcal{Y}}

\newcommand{\mfd}{\mathfrak{D}}

\newcommand{\mfa}{\mathfrak{a}}
\newcommand{\mfc}{\mathfrak{c}}

\newcommand{\mfg}{\mathfrak{g}}
\newcommand{\mfk}{\mathfrak{k}}
\newcommand{\mfm}{\mathfrak{m}}
\newcommand{\mfn}{\mathfrak{n}}
\newcommand{\mfr}{\mathfrak{r}}
\newcommand{\mfs}{\mathfrak{s}}
\newcommand{\mft}{\mathfrak{t}}

\newcommand{\msd}{\mathscr{D}}

\newcommand{\p}{\partial}

\newcommand{\mclogo}{{\,\,^{\circ\!\!\!\!\!}}{_0}\widehat{\mcm}_{A}}
\newcommand{\qmclogo}{{\,\,^{\circ\!\!\!\!\!}}{_0}\mcq\mcm_{A}}
\newcommand{\mclog}{{^{\circ\!\!}}\widehat{\mcm}_{A}}

\newcommand{\qma}{\mcq\mcm_A}
\newcommand{\wqma}{\widehat{\mcq\mcm}_A}

\DeclareMathOperator{\trTLEP}{\textup{trTLEP(n)}}
\DeclareMathOperator{\TEP}{\textup{TEP(n)}}
\DeclareMathOperator{\trTLEPlog}{\textup{log-trTLEP(n)}}
\DeclareMathOperator{\TEPlog}{\textup{log-TEP(n)}}

\def\cf{\textit{cf.}\kern.3em}

\def\resp{\textit{resp.}\kern.3em}
\renewcommand{\k}{\kern2pt}

\numberwithin{equation}{section} 

\newtheorem{thm}[equation]{Theorem}
\newtheorem{cor}[equation]{Corollary}
\newtheorem{lem}[equation]{Lemma}
\newtheorem{prop}[equation]{Proposition}
\newtheorem{defn}[equation]{Definition}
\theoremstyle{definition} 

\newtheorem{rem}[equation]{Remark}

\newtheorem{notn}[equation]{Notation}%
\newtheorem{assumption}[equation]{Assumption}%

\DeclareMathOperator{\FL}{FL}%
\DeclareMathOperator{\Hom}{Hom}%
\DeclareMathOperator{\Ext}{Ext}%
\DeclareMathOperator{\orb}{orb}%
\DeclareMathOperator{\Specm}{Specm}

\renewcommand{\k}{\kern2pt}

\begin{document}

\title{Logarithmic degenerations of Landau-Ginzburg models for toric orbifolds and global $ tt^*$ geometry}

\date{\today}

\author{Etienne Mann}

\address{Etienne Mann, D\'epartement de math\'ematiques,
B\^{a}timent I,
Facult\'e des Sciences,
2 Boulevard Lavoisier,
F-49045 Angers cedex 01,
France }
\email{etienne.mann@univ-angers.fr }
\urladdr{http://www.math.univ-angers.fr/~mann/}

 \author{Thomas Reichelt}
 \address{Thomas Reichelt,  Institut f\"{u}r Mathematik,
Universit\"{a}t Heidelberg, Im Neuenheimer Feld 205, 69120 Heidelberg, Germany}
\email{treichelt@mathi.uni-heidelberg.de}
\urladdr{http://mathi.uni-heidelberg.de/~treichelt/main-en.html}

\begin{abstract}
We discuss the behavior of Landau-Ginzburg models for toric orbifolds near the large volume limit. This enables us to express mirror symmetry as an isomorphism of Frobenius manifolds which aquire logarithmic poles along a boundary divisor. If the toric orbifold admits a crepant resolution we construct a global moduli space on the $B$-side and show that the associated $tt^*$-geometry exists globally. 
\end{abstract}


\thanks{E.M is supported by the grant of the Agence Nationale de la
    Recherche ``New symmetries on Gromov-Witten theories'' ANR- 09-JCJC-0104-01. Th.R. is supported by the Emmy-Noether DFG grant RE 3567/1-1\\
Both authors acknowledge partial support by the ANR/DFG grant ANR-13-IS01-0001-01/02, HE 2287/4-1 \& \\ SE 1114/5-1 (SISYPH)}

\maketitle
\setcounter{tocdepth}{3}     
\tableofcontents
\section{Introduction}
The present paper deals with classical Hodge-theoretic mirror symmetry for smooth toric Deligne-Mumford stacks. One of the first mathematical incarnations of this type of mirror symmetry  was a theorem of Givental identifying a solution (the so-called $J$-function) of the Quantum $\mcd$-module of a (complete intersection inside a) smooth toric variety with a  generalized hypergeometric function (the $I$-function). This has laid the foundation to express mirror symmetry as an equivalence of differential systems matching the Quantum $\mcd$-module on the $A$-side with certain (Fourier-Laplace transformed) Gau\ss-Manin systems coming from an algebraic family of maps (the Landau-Ginzburg model) on the $B$-side. An analytic proof of this fact using oscillating integrals was given by Iritani in \cite{Ir2}. A purely algebraic proof was given in \cite{RS1} where it was also shown  that the  Frobenius manifold an the $A$-side, which encodes the big quantum cohomology, is isomorphic to a Frobenius manifold on the $B$-side which comes from the Landau-Ginzburg model. The construction of Frobenius manifolds is a classical subject in singulartity theory. The first examples arose from the work of K.+ M. Saito on the base space of a semi-universal unfolding and later it was shown by Sabbah partly with  Douai  that these results carry over to an algebraic map which satisfies certain tameness assumptions. However, their construction is not unique in the sense that it depends on the choice of a good basis, which provides a solution to a Birkhoff problem, and on the choice of a primitive section. To circumvent this problem a careful analysis of the Fourier-Laplace transformed Gau\ss-Manin system and its degeneration along a boundary divisor, which contains the large volume limit, was carried out in \cite{RS1} in the case of Landau-Ginzburg models which serve as mirror partners for smooth nef toric varieties. Beyond the smooth case partial results for weighted projective spaces were obtained in \cite{Douai-Mann} where mirror symmetry is proven as an isomorphism of logarithmic Frobenius manifolds without pairing.\\ 

In this paper we prove mirror symmetry for smooth toric Deligne-Mumford stacks, satisfying a positivity condition, as an isomorphism of logarithmic Frobenius manifolds, which generalizes the theorem obtained in \cite{RS1} for smooth nef toric varieties.  In order to ensure a good  behavior of the connection and the pairing along the boundary divisor a careful choice of the coordinates on the complexified K\"ahler moduli space is needed. Since the Fourier-Laplace transformed Gau\ss-Manin system is a cyclic $\mcd$-module, the generator is a canonical candidate for the primitve section. The Birkhoff problem is solved at the large volume limit where we identify the fiber of the holomorphic bundle with the orbifold cohomology of the toric Deligne-Mumford stack.\\

The notion of $tt^*$ geometry was introduced by Cecotti and Vafa in their study of moduli spaces of $N=2$ supersymmetric quantum field theories. Hertling \cite{He4} formalized this structure under the name of pure and polarized TERP-structures and showed that the base space of a semi-universal unfolding of an isolated hypersurface singularity carried such a structure. In the case of a tame algebraic map, a theorem of Sabbah \cite{Sa2} shows that the corresponding Fourier-Laplace transformed Gau\ss-Manin system underlies a pure and polarized TERP-structure. In \cite{RS1} this was used to show that a Zariski open subset of the base space of the Landau-Ginzburg model carries a pure and polarized TERP-structure. Using mirror symmetry this induces $tt^*$ geometry on the quantum $\mcd$-module. Iritani \cite{Ir2} gave an intrinsic description of the real structure on the $A$-side using $K$-theory. In this paper the result of \cite{RS1} is generalized to toric orbifolds. If the toric orbifold $X$ admits a crepant resolution $Z$ we construct a global base space which contains two limit points corresponding to the large volume limit points of $X$ and $Z$ respectively. We prove that there exists $tt^*$-geometry on the whole moduli space which, when restricted to some analytic neighborhood of the large volume limits, is isomorphic to the $tt^*$-geometry coming from the quantum cohomology of $X$ resp. $Z$. The result here is in the sprit of Y.Ruan Crepant Transformation conjecture which has stimulated a lot of research:
\cite{MR2772168,MR2541935,MR2360646,MR2551767,MR2518631,MR2483931,
MR2411404,MR2357679,MR3202006,MR2553561,MR2529944,MR2486673,
MR3112518,MR2234886,2008PhDT.........4I,2011arXiv1109.5540L,
2013arXiv1311.5725L,2014arXiv1401.7097L}.\\

We give a short overview of the contents of this paper: In section 2 we recall some standard facts on toric Deligne-Mumford stacks. An import ingredient in the construction of the mirror Landau-Ginzburg model is the extended stacky fan of Jiang \cite{Jext}. This enables us to introduce the extended Picard group and the extended K\"ahler cone which are needed to construct coordinates on the base space of the mirror Landau-Ginzburg model. In section 3 we review the notion of the Fourier-Laplace transformed Gau\ss-Manin system and cite some results of \cite{RS1},\cite{RS2} which identifies the FL-Gau\ss-Manin system of a family of Laurent polynomials with a FL-transformed GKZ-system for which an explicit description as a cyclic $\mcd$-module is available. In the fourth section we use the results of the previous section to calculate the FL-transformed Gau\ss-Manin system corresponding to the Landau-Ginzburg model (cf. Proposition 4.4). As a next step we show that the FL-transformed Brieskorn lattice is coherent over the tame locus of the Landau-Ginzburg model (Theorem 4.10). We then analyze the degeneration behavior along a boundary divisor which contains the large volume limit. Finally we prove that there exists a canonical germ of a logarithmic Frobenius manifold associated to the Landau-Ginzburg model.
 Section 5 reviews orbifold quantum cohomology and the Givental connection. We show that the big quantum cohomology gives rise to a logarithmic Frobenius manifold (Proposition 5.7). In section 6, using a Givental-style mirror theorem of Coates, Corti, Iritani and Tseng \cite{CCIT}, we combine the last two sections to express mirror symmetry for toric Deligne-Mumford stacks as an isomorphism of logarithmic Frobenius manifolds (Theorem 6.6). In section 7 we consider a toric orbifold $X$ admitting a crepant resolution $Z$ and construct a global Landau-Ginzburg model. We prove that there exists a pure and polarized variation of TERP structures on the base space $\mcm$ of this model which gives the $tt^*$-geometry of the corresponding quantum $\mcd$-modules in different neighborhoods of $\mcm$.



\section{Some toric facts}\label{sec:some-toric-facts}

Let $G$ be a free abelian group. We associate to it the group ring $\mbc[G]$ which is generated by the elements $\chi^g$ for $g \in G$. Its maximal spectrum $\Specm(\mbc[G]) = Hom(G,\mbc^*)$ is naturally a commutative algebraic group (i.e. a torus). Let $\mfa: G \ra H$ be a group homomorphism between the free abelian groups $G$ and $H$. This induces a ring homomorphism
\begin{align}
\phi_\mfa: \mbc[G] &\lra \mbc[H]\, , \notag \\
\chi^g &\mapsto \chi^{\mfa(g)} \notag
\end{align}
and a morphism of algebraic groups
\begin{align}
\psi_\mfa: \Specm(\mbc[H]) \lra \Specm(\mbc[G])\, . \notag
\end{align}
Choose a basis $g_1,\ldots, g_n$ resp. $h_1,\ldots, h_m$ of $G$ resp. $H$. The homomorphism $\mfa$ is then given by a matrix $A = (a_{ij})$ with $\mfa(g_j) = \sum_{j=1}^m a_{ij} h_i$. The bases also determine coordinates $x_j = \chi^{g_j}$ resp. $y_i = \chi^{h_i}$ which identifies $\Specm(\mbc[G])$ with $(\mbc^*)^n$ and $\Specm(\mbc[H])$ with $(\mbc^*)^m$. In these coordinates the map $\psi_\mfa$ is given by
\begin{align}
(\mbc^*)^m &\lra (\mbc^*)^n\, , \notag \\
(y_1,\ldots, y_m) &\mapsto (\underline{y}^{\underline{a}_1},\ldots, \underline{y}^{\underline{a}_n})\notag
\end{align}
where $\underline{y}^{\underline{a}_j}:= \prod_{i=1}^n y_i^{a_{ij}}$.

\subsection{Toric Deligne-Mumford stacks}
A toric Deligne-mumford stack is constructed by a so-called stacky fan which was introduced by \cite{BCS}. A stacky fan 
\[
\mathbf{\Sigma} = (N, \Sigma, \mfa)
\]  
consists of  
\begin{itemize}
\item a finitely generated abelian group $N$ of rank $d$,
\item a complete simplicial fan $\Sigma$ in $N_\mbq := N \otimes_\mbz \mbq$, where we denote by $\Sigma(k)$ the set of $k$-dimensional cones of $\Sigma$ and by $\{\rho_1, \ldots , \rho_{m}\}$ the rays of $\Sigma$,
\item a homomorphism $\mfa: \mbz^{m} \ra N$ given by elements $a_1, \ldots a_{m}$ of $N$ with $a_i \in \rho_i$ and $\mfa(e_i)= a_i$, where $e_1, \ldots ,e_{m}$ is the standard basis of $\mbz^{m}$.
\end{itemize}
$ $\\
\textbf{Assumption}: In the rest of the paper we will assume that $N$ is torsion-free.\\

If we choose a basis $v_1,\ldots ,v_d$ of $N$ the map $\mfa$ is given by a matrix $A=(a_{ki})$.\\

The morphism $\mfa$ gives rise to a triangle in the derived category of $\mbz$-modules
\[
\mbz^m \overset{\mfa}\lra N \lra Cone(\mfa) \overset{+1}{\lra}\, .
\]
We apply the derived functor $RHom(-,\mbz)$ and consider the associated long exact sequence
\begin{equation}\label{eq:dualseqN}
0 \lra N^\star \overset{\mfa^\star}\lra (\mbz^m)^\star \lra \Ext^1(Cone(\mfa),\mbz) \lra 0\, ,
\end{equation}
where the injectivity of $\mfa^\star$ follows from the fact that the image of $\mfa$ has finite index in $N$ and the surjectivity of $(\mbz^m)^\star \lra Ext^1(Cone(\mfa),\mbz)$ follows from $Ext^1(N,\mbz) = 0$, i.e. from our assumption that $N$ is free.\\

Applying $Hom(-, \mbc^*)$  to the exact sequence \eqref{eq:dualseqN} gives the short exact sequence
\[
0 \lra G \overset{\psi_\mfa}{\lra} (\mbc^*)^m \lra Hom(N^\star, \mbc^*) \lra 0\, ,
\]
where $G:= Hom(\Ext^1(N,\mbz),\mbc^*)$. Here we have used the fact that $\mbc^*$ is a divisible group, hence $Hom(-, \mbc^*)$ is exact.\\

The set of anti-cones is defined to be
\[
\mca := \left\lbrace I \subset \{1, \ldots , m\} \mid \sum_{i \notin I} \mbq_{\geq 0} \rho_i\text{ is a cone in}\, \Sigma \right\rbrace\, .
\]

Each $I \in \mca$ gives rise to a subvariety $\mbc^I \subset \mbc^m$ given by  $\{(x_1,\ldots, x_m) \in \mbc^m \mid x_i = 0\; \text{for}\; i \notin I\}$. We set
\[
\mcu_\mca := \mbc^m \setminus \bigcup_{I \not \in \mca} \mbc^I\, .
\]

The toric Deligne-Mumford stack associated to this data is the following quotient stack:
\[
\mcx:= \mcx(\mathbf{\Sigma}) := [ \mcu_A /G ]\, ,
\]
where $G$ acts on $\mcu_\mca$ via $\psi_\mfa$.\\

For $\sigma \in \Sigma$, let
\[
\text{Box}(\sigma) = \{ a \in N \mid a = \sum_{a_i \in \sigma} r_i a_i , 0 \leq r_i  <1\}\, .
\]
We set 
\[
Box(\Sigma) =  \bigcup_{\sigma \in \Sigma} Box(\sigma)\, .
\]
The inertia stack $\mci \mcx(\mathbf{\Sigma})$ is the fiber product taken over the diagonal maps $\mcx \ra \mcx \times \mcx$. Its components are indexed by the set $Box(\Sigma)$:
\[
\mci \mcx(\mathbf{\Sigma}) = \bigsqcup \mcx_{(v)}\, ,
\]
where $\mcx_{(v)}$ is the toric orbifold $\mcx(\mathbf{\Sigma}/\sigma(v))$ with  $\sigma(v)$ being the smallest cone in $\Sigma$ which contains $v$ (cf. \cite{BCS} Section 4).\\

We have the following description of the orbifold cohomology ring of $\mathcal{X}$. As a $\mbq$-vector space it is isomorphic to the direct sum of the cohomology groups of the components of its inertia stack:
\[
H^*_{orb}(\mcx,\mbq) \simeq \bigoplus_{v \in Box(\Sigma)} H^{*-2 i_v}(\mcx_{(v)},\mbq)\, ,
\]
where $i_v:= \sum r_i$ for $v \in Box(\Sigma)$.  The orbifold cohomology of $\mcx$ carries a product which makes $ H^*_{orb}(\mcx,\mbq)$  into a graded algebra.  A combinatorial description in terms of the fan has been given by Borisov, Chen and Smith \cite{BCS} and, in the semi-projective case, by Jiang and Tseng \cite{JT}.  Let $N_{\Sigma} := \{ c \in N \mid c \in |\Sigma| \}$ and let
\[
\mbq[N_\Sigma] := \bigoplus_{c \in N_\Sigma} \mbq \chi^c
\]
with the product
\[
\chi^{c_1} \cdot \chi^{c_2} := \begin{cases} \chi^{c_1+c_2} & \text{if there exists $\sigma \in \Sigma$ such that $c_1, c_2 \in \sigma$,} \\ 0 & \text{otherwise}\, . \end{cases}
\]
Let  $c \in N_{\Sigma}$ and $\sigma(c)$ be the minimal cone containing $c$. Then $c$ can be uniquely expressed as
\[
c = \sum_{a_i \in \sigma(c)} r_i a_i \,.
\]
We define 
\[
deg(\chi^c) := deg(c) := \sum r_i\, .
\]
Using this graduation $\mbq[N_\Sigma]$ becomes a graded semigroup ring. By \cite{BCS} we have the following isomorphism of $\mbq$-graded rings:
\begin{equation}\label{eq:BCSiso}
H^*_{orb}(\mcx,\mbq) \simeq \frac{\mbq[N_\Sigma]}{\{\sum_{i=1}^m \kappa(a_i) \chi^{a_i} \mid \kappa \in N^\ast\}}\, .
\end{equation}

Denote by $PL(\Sigma)$ the free $\mbz$-module of continuous piece-wise linear functions on $\Sigma$ having integer values on $N$. We have the natural embedding of $N^\ast = \Hom(N,\mbz)$ into $PL(\Sigma)$, where the cokernel of this map is isomorphic to the Picard group of the underlying coarse moduli space $X:=X(\Sigma)$:
\begin{equation}\label{eq:PLPic}
0 \ra N^\ast \ra PL(\Sigma) \ra Pic( X) \lra 0\, .
\end{equation}

We have isomorphisms
\[
\xymatrix{ 0 \ar[r] & N^\star \otimes \mbq  \ar[r] & PL(\Sigma) \otimes \mbq \ar[r] & Pic(X)\otimes \mbq \ar[r] & 0 \\
0 \ar[r] & N^\star \otimes \mbq \ar[r] \ar[u]^\simeq & \mbq^m \ar[u]^\simeq  \ar[r] & H^2(X,\mbq) \ar[r] \ar[u]^\simeq & 0}
\]
where the image of the standard generator $\overline{D}_i \in \mbq^m$ in $PL(\Sigma) \otimes \mbq$ is the piece-wise linear function  having value $1$ on $a_i$ and $0$ on $a_j$ for $j \in \{1,\ldots,m\} \setminus \{i\}$. We denote the image of $\overline{D}_i$ in $Pic(X) \otimes \mbq$ by $[\overline{D}_i]$.

\subsection{Extended stacky fans}\label{subsec:Exstacky}
Toric Deligne-Mumford stacks can also be described by a so-called extended stacky fan (cf. \cite{Jext}. To the datum of a stacky fan $\mathbf{\Sigma} = (N, \Sigma, \mfa)$ one adds  a map $\mbz^e \ra N$ and writes $S=\{a_{m+1}, \ldots , a_{m+e}\}$  for the image of the standard basis. By abuse of notation  we will call the following map still $\mfa$ 
\begin{align}
\mfa: \mbz^{m+e} &\lra N\, , \notag \\
e_i &\mapsto \mfa(e_i) =  a_i \quad \text{for}\; i = 1, \ldots m+e\, . \notag
\end{align}
The $S$-extended stacky fan $\mathbf{\Sigma}^e = ( N, \Sigma, \mfa)$ is given by the free group $N$, the fan $\Sigma$ and the map $\mfa: \mbz^{m+e} \ra N$.\\

Assumption: In the following we will choose $a_{m+1}, \ldots, a_{m+e}$ in such a way so that $\mfa$ is surjective.\\
 
We denote by $\mbl$ the kernel of $\mfa$. This gives us as above the two exact sequences
\begin{align}
&0 \lra \mbl  \lra \mbz^{m+e} \overset{\mfa}{\lra} N \lra 0\, , \label{eq:extseq} \\
&0 \lra N^\star \lra (\mbz^{m+e})^\star \lra \mbl^\star \lra 0\, . \label{eq:extseqdual}
\end{align}

We denote by $D_1,\ldots, D_{m+e}$ the standard basis of $(\mbz^{m+e})^\star$, i.e. $ D_i(e_j) = \delta_{ij}$ and by  $[D_1],\ldots , [D_{m+e}]$ the images of $D_1,\ldots, D_{m+e}$ in $\mbl^\star$.\\

Applying the exact functor $Hom(-,\mbc^*)$ to the sequence $\ref{eq:extseqdual}$ gives a map 
\[
\psi_{\mfa}: G^e \ra (\mbc^*)^{m+e},
\]
where $G^e := Hom_\mbz(\mbl^\star, \mbc^*)$. We set
\[
\mcu_{\mca}^e := \mcu_\mca \times (\mbc^*)^{e}\, ,
\]
then $G^e$ acts on $\mcu_A^e$ via $\psi_{\mfa}$  and the quotient stack
\[
[\mcu_\mca^e / G^e]
\]
is isomorphic to the stack $\mcx(\mathbf{\Sigma})$ by \cite{Jext}.\\

Given a stacky fan $\mathbf{\Sigma} = (N, \Sigma, \mfa)$ there exists a ``canonical'' choice of an extended stacky fan $\mathbf{\Sigma}^e$. Let $Gen(\sigma)$ be the subset of $Box(\sigma)$ of elements which are primitive in $\sigma \cap N$, i.e. which can not be generated by other elements in the semigroup $\sigma \cap N$ and set $Gen(\Sigma) := \bigcup_{\sigma \in \Sigma} Gen(\sigma)$.\\

In the following we will always choose 
\[
S = Gen(\Sigma)
\]
and we will set $n:= m+e$ and $\mcg:= \{a_1, \ldots , a_m\} \cup Gen(\sigma)$. Notice that $\{a_1, \ldots , a_m\} \cap Gen(\sigma) = \emptyset$, hence the cardinality of $\mcg$ is $n$.\\

This choice of an extended stacky fan will allow us to give a different description of the orbifold cohomology ring which will be very useful for our purposes.

We introduce for every $a_i \in \mcg$ a formal variable $\mathfrak{D}_i$. For a top-dimensional cone $\sigma \in \Sigma(d)$ define in $\mbc[\mfd_1, \ldots , \mfd_n]$ the ideal
\[
\mcj(\sigma) := \left\langle \prod_{l_i >0, a_i \in \sigma} \mfd_i^{l_i} - \prod_{l_i< 0 , a_i \in \sigma} \mfd_i^{-l_i} \mid  \sum_{a_i \in \sigma} l_i a_i = 0 , l_i \in \mbz\right\rangle\, .
\]

We call relations $\underline{l}=(l_1, \ldots , l_n)$ of such type \textbf{cone relations}. The ideal
\[
\mcj(\Sigma_X) := \sum_{\sigma \in \Sigma_X(d)} \mcj(\sigma)
\]
is called the cone ideal of $\Sigma$.

Let $\mck(\Sigma)$ be the ideal which is generated by
\[
E_k := \sum_{i=1}^{m} a_{ki} \mfd_i \qquad \text{for} \quad k = 1, \ldots ,d\, ,
\]
where $a_{ki}$ is the $k$-th coordinate of $a_i \in \mbz^d$.\\

We call $I \subset \{1, \ldots ,n\}$ a \textbf{generalized primitive collection} if the set $\{a_i \mid i \in I\}$ is not contained in a cone of $\Sigma$ and if any proper subset of $\{a_i \mid i \in I\}$ is contained in some cone of $\Sigma$. We denote by $\mcg \mcp$ the set of generalized primitive collections.\\

The orbifold cohomology of $\mcx$ can then be expressed in the following way. 

\begin{lem}[\cite{TW} Lemma 2.4]\label{lem:orbcoho}
Let $deg(\mfd_i) =deg(a_i) $ for $i=1, \ldots , n$, then we have an isomorphism of graded $\mbc$-algebras
\[
H^*_{orb}(\mcx) \simeq \frac{\mbc[\mfd_1, \ldots , \mfd_{n}]}{\mcj(\Sigma_X) + \mck(\Sigma_X) + \langle \prod_{i\in I} \mfd_i \mid I \in \mcg \mcp \rangle}\, ,
\]
which sends $\mfd_i$ to $\chi^{a_i}$ for $i=1,\ldots n$.
\end{lem}

\begin{rem}\label{rem:classcoho}
We would like  to remind the reader that the ordinary cohomology ring of the underlying coarse moduli space $X$ is given by
\[
H^*(X,\mbc) \simeq \frac{\mbc[\mfd_1,\ldots, \mfd_m]}{\mck(\Sigma)+  \langle \prod_{i\in I}\mfd_i  \mid I \in \mcp \rangle}
\]
where $\mcp$ is the set of primitive collections. Here a collection $I \subset \{1,\ldots,m\}$ is primitive  if the set $\{a_i \mid i \in I\}$ does not span a cone of $\Sigma$ but any proper subset does.
\end{rem}

\subsection{The extended Picard group}
We have the following commutative diagram (cf. \eqref{eq:PLPic})

\[
\xymatrix{ 0 \ar[r]&  N^\star \ar[r] & (\mbz^{n})^\star \ar[r] & \mbl^\star \ar[r] &0 \\
0 \ar[r] & N^\star \ar[r] \ar@{=}[u]& PL(\Sigma) \ar[r] \ar[u]^\Theta & Pic(X) \ar[r] \ar[u] & 0 }
\]
where the map $PL(\Sigma) \ra (\mbz^n)^\star$ is given by 

\begin{align}
\Theta: PL(\Sigma) &\lra (\mbz^{n})^\star\, , \label{def:Theta} \\
\varphi &\mapsto (\varphi(a_1), \ldots , \varphi(a_{n}))\, . \notag 
\end{align}

We want to determine the image of this map. For this, let 
\[
I_\sigma := \{i \in\{1,\ldots,m\} \mid a_i \in \sigma\}\, .
\]
We consider the distinguished relations 
\begin{equation}\label{eq:disrel}
a_{m+k} -\sum_{i \in I_\sigma} r_{kj} a_i = 0 \qquad \text{for}\; k=1,\ldots ,e\, ,
\end{equation}
which give elements $l_1,\ldots , l_e \in \mbl \otimes \mbq$.\\

\begin{lem}
The image of $\Theta$ is as saturated subgroup of $(\mbz^n)^\star$, i.e.
\[
(\Theta(PL(\Sigma))\otimes_\mbz \mbq) \cap (\mbz^{n})^\star =  \Theta(PL(\Sigma))\, .
\]
\end{lem}
\begin{proof}
Denote by $\mathbf{K}$ 
the kernel of the map
\begin{align}
(\mbz^{n})^\star &\lra \mbq^e \notag \\
x &\mapsto (x(l_1),\ldots, x(l_e)) \notag
\end{align}
In order to show the claim it is enough to show the following equality
\[
\mathbf{K} = \Theta(PL(\Sigma))\, . 
\]
It is clear that $\Theta(PL(\Sigma)) \subset \mathbf{K}$ 
since  a function in $PL(\Sigma)$ is linear on each cone $\sigma \in \Sigma$. Now let $u \in \mathbf{K}$ 
and let $\sigma \in \Sigma(d)$ be a maximal cone. Choose a $\mbz$-basis $a_{j_1}, \ldots , a_{j_d}$ in the set $Gen(\sigma) \cup \{a_j \mid a_j \in \sigma\}$. The values $u(a_{j_1}), \ldots , u(a_{j_d})$ determine an element  $m_\sigma \in N^\vee$. Since $u \in \mathbf{K}$ 
we get $u(a_i) = m_\sigma(a_i)$ for $a_i \in Gen(\sigma) \cup \{a_j \mid a_j \in \sigma\}$. Repeating this for any cone we get an element in $PL(\Sigma)$ whose image under $\Theta$ is $u$.
\end{proof}

We denote the full sublattice of $(\mbz^{n})^\star$ generated by $\Theta(PL(\Sigma))$ and $D_{m+1},\ldots , D_{m+e}=D_n$ by $PL(\Sigma^e)$ and call its image in $\mbl^\star$ the extended Picard group $Pic^e(X)$.\\

We get an exact sequence of $\mbz$-free modules
\begin{equation}
0 \lra N^* \lra PL(\Sigma^e) \lra Pic^e(X) \lra 0\, .
\end{equation}

The map $\Theta$ induces a map 
\[
\theta: Pic(X) \lra Pic^e(X).
\]
Notice that the images of $[\overline{D}_i] \in Pic(X) \otimes \mbq$ for $i=1,\ldots ,m$ in $Pic^e(X) \otimes \mbq$ are given by 
\[
\theta([\overline{D}_i]) =  [D_i]+ \sum_{k=1}^{e}  c_{ki} [D_{m+k}]\, ,
\]
which follows from the formulas \ref{def:Theta} and \ref{eq:disrel}.
\subsection{The extended K\"{a}hler cone}
Inside $PL(\Sigma)$ we consider the cone of convex functions $CPL(\Sigma)$. It has non-empty interior since $X$ is projective. We denote its $\mbq_{\geq 0}$-span in $H^2(X,\mbq)$ by $\overline{\mck}$. Consider now the cone $CPL(\Sigma^e)$ generated by $\Theta(CPL(\Sigma))$ and $D_{m+1}, \ldots , D_{m+e}$.  We denote the $\mbq_{\geq 0}$-span of $\Theta(CPL(\Sigma))$ resp. $CPL(\Sigma^e)$ in $\mbl^\star\otimes \mbq$ by $\mck$ resp. $\mck^e$ and call it the K\"ahler resp. extended K\"ahler cone.

We denote the image of anti-canonical divisor $-K_X$ of $X$ in $Pic(X(\Sigma) \otimes \mbq \simeq H^2(X,\mbq)$ by $\overline{\rho}$. It is given by $\overline{\rho} = [\overline{D}_1]+ \ldots [\overline{D}_m]$. The toric variety $X$ is weak Fano if $\overline{\rho} \in \overline{\mck}$. Consider the following class in $\mbl^\star$
\[
\rho:= [D_1]+\ldots + [D_{m+e}]\, .
\]
Later we will impose the following condition
\[
\rho \in \mck^e\, .
\]
There is the following characterization of this condition
\begin{lem}\label{lem:degai}
We have $\rho \in \mck^e$ iff $\overline{\rho} \in \overline{\mck}$ and $deg(a_i) \leq 1$ for $i=m+1,\ldots,m+e$.
\end{lem}
\begin{proof}
The element $\rho$ can be expressed in the following way
\begin{align}
\rho = [D_1]+ \ldots + [D_{m+e}] &= \theta([\overline{D}_1])+\ldots + \theta([\overline{D}_m])+ \sum_{k=1}^{e}(1- \sum_{i=1}^{m}r_{ki})[D_{m+k}]\notag \\
&= \theta([\overline{D}_1]) + \ldots + \theta([\overline{D}_m]) + \sum_{k=1}^e(1-deg(a_{m+k}))[D_{m+k}] \notag \\
& \in \qquad \qquad \mck  \qquad \qquad \quad\;\; \oplus  \qquad \qquad\bigoplus_{k=1}^e\mbq [D_{m+k}]\, . \notag
\end{align}
The last term is in $\mck^e$ iff $deg(a_{m+k}) \leq 1$ for $k =1,\ldots ,e$.
\end{proof}

Notice that the degree function $deg$ gives rise to a piece-wise linear function $\varphi$ which is given by 
\begin{equation}
\varphi(a_i) = 1  \qquad \text{for} \quad i=1,\ldots, m\, . \label{def:varphi}
\end{equation}
This piece-wise linear function corresponds to the anti-canonical divisor. If we assume that $X$ is nef (i.e. $\overline{\rho} \in \overline{\mck}$) then $\varphi$ is a convex function.
\begin{rem}\label{rem:hleq2eqle}
It follows from Lemma \ref{lem:orbcoho} and Lemma \ref{lem:degai} that for the choice $S = Gen(\Sigma)$ we have an isomorphism $H^{\leq 2}_{orb}(\mcx,\mbq) \simeq \mbl^\star \otimes \mbq$.
\end{rem}

We now introduce  the so-called extended Mori cone. Set
\[
\mca^e := \{ I \cup \{m+1,\ldots, m+e\} \mid I \in \mca\}
\]
and
\begin{align}
\mbk &:= \{ d \in \mbl \otimes \mbq \mid \{i \in \{1,\ldots, m+e\} \mid \langle D_i,d\rangle \in \mbz\} \in \mca^e\}\, , \notag \\
\mbk^{\text{eff}} &:= \{ d \in \mbl \otimes \mbq \mid \{i \in \{1,\ldots, m+e\} \mid \langle D_i,d\rangle \in \mbz_{\geq 0}\} \in \mca^e\}\, . \notag
\end{align}
Notice that the lattice $\mbl$ acts on $\mbk$.

Denote by $\lceil \cdot \rceil$, $\lfloor \cdot \rfloor$ and $\{ \cdot \}$ the ceiling, floor and  fractional part of a real number.
\begin{lem}[\protect{\cite[Section 3]{Ir2}}]\label{lem:kltobox}
The map
\begin{align}
\mbk / \mbl &\lra  Box(\Sigma)\, , \notag  \\
d & \mapsto v(d):=\sum_{i=1}^{m+e} \lceil \langle D_i,d\rangle \rceil a_i \notag
\end{align}
is bijective.
\end{lem}
\begin{proof}
We first notice that $ \sum_{i=1}^{m+e} \lceil \langle D_i,d\rangle \rceil a_i\in N$. From the definition of $\mbk$  and the exact sequence \ref{eq:extseq} we get
\begin{align}
v(d) &= \sum_{i=1}^{m+e} (\{-\langle D_i,d\rangle \} + \langle D_i,d \rangle)a_i 
= \sum_{i=1}^{m+e} \{-\langle D_i,d\rangle \}a_i 
= \sum_{a_i \in \sigma} \{-\langle D_i,d\rangle \}a_i \label{eq:defvd}
\end{align}
for some $\sigma$. This shows $v(d) \in Box(\sigma)$. From the formula \ref{eq:defvd} we easily see that the map $d \mapsto v(d)$ factors through $\mbk \ra \mbk/\mbl$. Choose $v \in Box(\Sigma)$. We can express $v$ either as $v = \sum_{a_i \in \sigma} n_i a_i$ with $n_i \in \mbn$ since $S = Gen(\Sigma)$ or as $v = \sum_{i \in I_{\sigma}} r_i a_i \in Box(\Sigma)$ with $r_i \in [0,1)$. The equation $\sum_{a_i \in \sigma} n_i a_i - \sum_{i \in I_{\sigma}} r_i a_i =0$ gives rise to an element in $\mbk \subset \mbl \otimes  \mbq$, this shows the surjectivity. In order to show injectivity let $d,d' \in \mbk$ with $v(d) = v(d')$. This means there exists a $\sigma,\sigma' \in \Sigma$ such that
\[
\sum_{a_i \in \sigma} \{-\langle D_i,d\rangle\}a_i = v(d) = v(d')=\sum_{a_i \in \sigma'} \{-\langle D_i,d\rangle\}a_i
\]
since both cones are simplicial we find a cone $\sigma'' \subset \sigma \cap \sigma'$ such that
\[
v(d) =  \sum_{a_i \in \sigma ''} \{-\langle D_i,d \rangle\}a_i = \sum_{a_i \in \sigma ''} \{-\langle D_i,d' \rangle\}a_i = v(d')
\]
and therefore $\{-\langle D_i,d\rangle \} = \{-\langle D_i,d'\rangle \}$ for all $i=1,\ldots m+e$. Hence $\{\langle D_i, d-d'\rangle \}= 0$ and there fore $d-d' \in \mbl$. This shows the injectivity.
\end{proof}

Denote by $NE^e(X)\subset \mbl \otimes \mbq$ the lattice dual to the extended Picard group $Pic^e(X) \subset \mbl^\star$ which gives the following short exact sequence of free $\mbz$-modules
\begin{equation}\label{eq:exseqextN}
0 \lra NE^e(X) \lra PL(\Sigma^e)^* \lra N \lra 0\, .
\end{equation}

\begin{lem}
There are the following inclusions
\[
\mbl \subset \mbk \subset NE^e(X)\, .
\]
\end{lem}
\begin{proof}
Notice that the first inclusion follows from the fact that $\{1,\ldots,m+e\} \in \mca^e$.
For each $v \in Box(\Sigma)$  we can write as above $v = \sum_{a_i \in \sigma} n_i a_i = \sum_{i \in I_\sigma} r_i a_i$. Denote by $d_v \in \mbl \otimes \mbq$ the element in $\mbk \subset \mbl \otimes \mbq$ which corresponds to the relation $\sum_{a_i \in \sigma}n_i a_i - \sum_{i \in I_\sigma} r_i a_i=$. The proof of Lemma \ref{lem:kltobox} shows that $\mbk$ is the union of the sets $d_v + \mbl$ for $v \in Box(\Sigma)$. So in order to prove the second inclusion it is enough to show that $d_v \in NE^e(X)$.
 For this we need to check that $L(d_v) \in \mbz$ for every $L \in Pic^e(X)$. Recall that $Pic^e(X)$ is the image $PL(\Sigma^e) \subset (\mbz^{n})^\star$ generated by $\Theta(PL(\Sigma))$ and $D_{m+1},\ldots, D_{m+e}$. 
 Take a lift $\varphi^e= \Theta(\varphi)+\sum_{i=m+1}^{m+e}t_i D_i=\sum_{i=1}^m \varphi(a_i) D_i + \sum_{i=m+1}^{m+e}t_i D_i$ of $L$ in $PL(\Sigma^e)$. We have 
\[ 
L(d_v) =  \varphi^e(\sum_{a_i \in \sigma}n_i e_i - \sum_{i \in I_\sigma} r_i e_i) = \sum_{i=1}^{m}n_i \varphi(a_i) + \sum_{i=m+1}^{m+e}t_i n_i - \varphi(\sum_{i \in I_\sigma} r_i a_i) \in \mbz
\]
since $\varphi$ is integer-valued on $N$.
\end{proof}
\section{Laurent polynomials and GKZ-systems}\label{sec:LauGKZ}

In this section we review some results from \cite{RS1} and \cite{RS2} concerning the relationship between (Fourier-Laplace-transformed) Gau\ss-Manin systems of families of Laurent polynomials and (Fourier-Laplace-transformed) GKZ-systems.

\begin{notn}
We will first review some notations from the theory of algebraic $\mcd$-modules. Let $X$ be a smooth algebraic variety over $\mbc$ of dimension $d$.  We denote by $\mcd_X$ the sheaf of algebraic differential operators and by $D_X = \Gamma(X,\mcd_X)$ its sheaf of global sections. Recall when $X$ is affine there is an equivalence of categories between $\mcd$-modules on $X$ which are quasi-coherent as $\mco_X$-modules and the corresponding $D_X$-module of global sections. If $\mcm$ is a $\mcd$-module on $X$ we will write $M$ for its module of global sections.  We denote by $M(\mcd_X)$ the abelian category of of algebraic $\mcd_X$-modules and the abelian subcategory of (regular) holonomic $\mcd_X$-modules by $M_h(\mcd_X)$ (resp. $M_{rh}(\mcd_X)$. The full triangulated subcategory of $D^b(\mcd_X)$ which consists of objects with (regular) holonomic cohomology is denoted by $D^b_h(\mcd_X)$ ( resp. $D^b_{rh}(\mcd_X)$). Let $f: X\ra Y$ be a map between smooth algebraic varieties and let $\mcm \in D^b(\mcd_X)$ and $\mcn \in D^b(\mcd_Y)$. The direct resp. inverse image functors are defined by $f_+ \mcm := Rf_*(\mcd_{Y \leftarrow X} \overset{L}\otimes \mcm)$ resp. $f^+ \mcn := \mcd_{X \ra Y} \overset{L}\otimes f^{-1} \mcn$.\\

Let $\mcv' := \mbc_t \times X$ be a trivial vector bundle of rank one and denote by $\mcv =\mbc_\tau \times X$ its dual. Denote by $can: \mcv' \times_X \mcv \ra \mbc$ the canonical pairing between its fibers.
\begin{defn}
Let $\mcl := \mco_{\mcv' \times_X \mcv} e^{-can}$ the free rank one module with differential given by the product rule. Denote by $p_1: \mcv' \times_X \mcv \ra \mcv'$, $p_2: \mcv' \times_X \mcv \ra \mcv$ the canonical projections. The Fourier-Laplace transformation is defined by
\[
FL_X(\mcm):= p_{2+}(p_1^+ \mcm \overset{L}\otimes \mcl) \quad \text{for} \quad M \in D^b_h(\mcd_{\mcv'})\, .
\]
\end{defn}
Set $ z = 1/\tau$ and denote by $j_\tau: \mbc_\tau^\ast \times X \hookrightarrow \mbc_\tau \times X$ and $j_z: \mbc^*_\tau \times X \hookrightarrow \hat{V} := \mbc_z \times X = \mbp^1_\tau \setminus \{ \tau = 0\} \times X$ the canonical embeddings. The partial, localized  Fourier-Laplace transformation is defined by
\[
Fl^{loc}_X(\mcm) := j_{z+}j_\tau^+FL_X(\mcm) \qquad \text{for} \quad \mcm \in D^b_h(\mcd_{\mcv'})\, .
\]
\end{notn}

Set $\hat{V} := \mbc_z \times \Lambda$ , where $\Lambda = \mbc^n$ with coordinates $\lambda_1, \ldots , \lambda_n$. Let $A$ be a $d \times n$ integer matrix with columns $(\underline{a}_1, \ldots , \underline{a}_n)$ and entries $a_{ki}$ for $k=1, \ldots , d$, $i= 1, \ldots , n$ and $\beta = (\beta_1, \ldots , \beta_d) \in \mbc^d$. We denote by $\mbl \subset \mbz^{n}$ the $\mbz$-submodule of relations among the columns $A$, i.e. $(l_1, \ldots , l_n) \in \mbl $ iff $\sum_i l_i \underline{a}_i = 0$.

\begin{defn}\label{def:FLGKZ}
The Fourier-Laplace-transformed GKZ-system $\widehat{M}_A^{(\beta_0,\beta)}$ is the left $D_{\hat{V}}$-module $D_{\hat{V}}[z^{-1}]/I$, where $I$ is the left ideal generated by the operators $\hat{\Box}_{\underline{l}}$, $\hat{E}_k -\beta_kz$ and $\hat{E} - \beta_0 z$, which are defined by
\begin{align}
\hat{\Box}_{\underline{l}} &:= \prod_{i: l_i <0} (z \cdot \p_{\lambda_i})^{-l_i} - \prod_{i : l_i >0} (z \cdot \p_{\lambda_i})^{l_i}\quad \text{for}\; \underline{l} \in \mbl \, ,\notag \\
\hat{E} &:= z^2 \p_z + \sum_{i=1}^m z \lambda_i \p_{\lambda_i}\, , \notag \\
\hat{E}_k &:= \sum_{i=1}^m a_{ki} z \lambda_i \p_{\lambda_i} \, . \notag
\end{align}
We denote the corresponding $\mcd_{\hat{V}}$-module by $\widehat{\mcm}^{(\beta_0, \beta)}_A$.
\end{defn}

Let $Y = (\mbc^*)^d$, we define a related family of Laurent polynomials:
\begin{align}
\varphi_A = (\phi_A, pr_2) : Y \times \Lambda &\lra V := \mbc_{\lambda_0} \times \Lambda\, , \notag \\
(\underline{y}, \lambda_1, \ldots,\lambda_n) &\mapsto (- \sum^n_{i=1} \lambda_i \underline{y}^{\underline{a}_i}, \lambda_1, \ldots, \lambda_n)\, . \notag
\end{align}

The Gau\ss-Manin system is the zeroth cohomology of the direct image of the  structure sheaf $\mco_{Y \times \Lambda}$ in the category of $\mcd$-modules:
\[
\mch^0(\varphi_{A,+} \mco_{Y \times \Lambda})\, .
\]

We now consider the localized partial Fourier-Laplace transform of the Gau\ss-Manin system of $\varphi_A$:
\[
\mcg^+ := FL^{loc}_\Lambda\mch^0(\varphi_{A,+} \mco_{Y \times \Lambda})\, .
\]
Write $G^+:=H^0(\widehat{V}, \mcg^+)$ for its module of global sections.
Then there is an isomorphism of $D_{\widehat{V}}$-modules (cf. e.g. \cite[Lemma 3.4]{RS2})
$$
G^+\cong H^0\left(\Omega^{\bullet+d}_{Y \times \Lambda / \Lambda}[z^\pm],d - z^{-1} \cdot d_y \phi\wedge\right),
$$
where $d$ is the differential in the relative de Rham complex $\Omega^\bullet_{Y \times \Lambda/\Lambda}$.

The following result relates the localized FL-transform of the Gau\ss-Manin system of $\varphi_A$  with a certain FL-transformed GKZ-system.

\begin{prop}
Assume $\mbr_+ A = \mbr^d$, then we have an isomorphism
\[
\mcg^+ \simeq \widehat{\mcm}^{(0,0)}_A\, .
\]
\end{prop}
\begin{proof}
This follows from \cite[Proposition 3.3]{RS2} and the assumption.
\end{proof}
In the following we set $\widehat{\mcm}_A := \widehat{\mcm}^{(0,0)}_A$ resp. $\widehat{M}_A := \widehat{M}^{(0,0)}$.\\

For certain parameters $\lambda \in \Lambda$ the fibers of $\varphi_A(\cdot,\lambda)$ acquire singularities at infinity. Outside this set the singularities of the $\mcd$-module $\mcg^+$ are particularly simple. \\

Let $Q$ be the convex hull of the set $\{\underline{a}_0 := 0, \underline{a}_1, \ldots , \underline{a}_n\}$:
\[
Q := conv(0,\underline{a}_1,\ldots ,\underline{a}_n)\, .
\]
 Let $\Gamma$ be a face of $Q$ and denote by $Y_\Gamma^{crit,(\lambda_0,\underline{\lambda})}$ the set
\[
\{ (y_1, \ldots ,y _d) \in Y \mid \sum_{\underline{a}_i \in \Gamma} \lambda_i \underline{y}^{\underline{a}_i} = 0 \; ; \; y_k \p_{y_k}(\sum_{\underline{a}_i \in \Gamma} \lambda_i \underline{y}^{\underline{a}_i}) = 0 \quad \text{for all} \quad k \in \{1, \ldots ,d\} \}\, .
\]
We say that the fiber $\varphi_A^{-1}(\lambda_0,\underline{\lambda})$ has a singularity at infinity if there exists a \textbf{proper} face $\Gamma$ of the Newton polyhedron $Q$ such that $Y_\Gamma^{crit, (\lambda_0,\underline{\lambda})} \neq \emptyset$. The set
\[
\Delta^\infty_A := \{ (\lambda_0,\underline{\lambda}) \in \mbc_{\lambda_0} \times \Lambda \mid \exists \Gamma \neq Q \; \text{such that } \; Y_{\Gamma}^{crit,(\lambda_0,\underline{\lambda})} \neq \emptyset \}
\]
is called the non-tame locus of $\varphi_A$. Notice that $\Delta^\infty_A$ is independent of $\lambda_0$ since $0$ lies in the interior of $Q$, hence no proper face of $Q$ contains $0$. Denote the projection of $\Delta^\infty_A$ to $\Lambda$ by $\Lambda^{bad}$. Let $\Lambda^* := \Lambda \setminus \{\lambda_1\cdot  \ldots \cdot \lambda_n = 0\} $ and define
\[
\Lambda^\circ := \Lambda^* \setminus \Lambda^{bad}\, .
\]
The following was  proven in \cite{RS2} Lemma 3.13:

\begin{lem}
Consider $\widehat{\mcm}_A$ as a $\mcd_{\mbp^1\times \overline{\Lambda}}$-module, where $\overline{\Lambda}$ is a smooth projective compactification of $\Lambda$. Then $\widehat{\mcm}_A$ is regular outside $(\{z=0\} \times \Lambda) \cup (\mbp^1_z \times (\overline{\Lambda} \setminus \Lambda^\circ))$ and smooth on $\mbc^*_z \times \Lambda^\circ$.
\end{lem}
Next we want to consider natural lattices in $\widehat{\mcm}_A$. For this we need the notion of $\mcr$-modules. 
\begin{defn}
Let $X$ be a smooth variety . Then the sheaf of (non-commutative) rings $\mcr_{\mbc_z \times X}$ is by definition the $\mco_{\mbc_z \times X}$-subalgebra  of $\mcd_{\mbc_z \times X}$ locally generated by $z^2 \p_z, z \p_{x_1}, \ldots , z\p_{x_n}$, where $(x_1, \ldots , x_n)$ are local coordinates on $X$.
\end{defn}

\begin{defn}
$ $\\[-1 em]
\begin{enumerate} 
\item Let $\mci$ be the left ideal in $\mcr_{\mbc_z \times \Lambda^*}$ generated by $(\hat{\Box}_{\underline{l}})$ and $(\hat{E}_k)_{k=0, \ldots ,d}$ (cf . Definition \ref{def:FLGKZ}). Write ${\,\,^{\ast\!\!\!\!\!}}{_0}\widehat{\mcm}_{A}$ for the cyclic $\mcr$-module $\mcr_{\mbc_z \times \Lambda^*} / \mci$.
\item Consider the open inclusions $\Lambda^\circ \subset \Lambda^* \subset \Lambda$ and define the $\mcd_{\mbc_z \times \Lambda^\circ}$-module
\[
\mclog := \left( \widehat{\mcm}_A \right)_{\mid \mbc_z \times \Lambda^\circ}
\]
and the $\mcr_{\mbc_z \times \Lambda^\circ}$-module
\[
\mclogo := \left( {\,\,^{\ast\!\!\!\!\!}}{_0}\widehat{\mcm}_{A} \right) _{\mid \mbc_z \times \Lambda^\circ}  \, .
\]
\end{enumerate}
\end{defn}

We now list some properties of the $\mcr$-module $\mclogo$:
\begin{prop}\cite[Proposition 3.18, Corollary 3.19]{RS2}\label{prop:PropLattice}$ $\\[-1em]
\begin{enumerate}
\item The $\mco_{\mbc_z \times \Lambda^\circ}$-module $\mclogo$ is locally-free of rank vol(Q).
\item The natural map $\mclogo \ra \mclog$ which is induced by the inclusion $\mcr_{\mbc_z \times \Lambda^\circ} \ra \mcd_{\mbc_z \times \Lambda^\circ}$ is injective.
\end{enumerate}
\end{prop}

The so-called Fourier-Laplace transformed Brieskorn lattice of the FL-transformed Gau\ss-Manin system $\mcg^+$ is given by the following $R_{\mbc_z \times \Lambda^\circ}$-module:
\[
H^0\left(\Omega^{\bullet +d}_{Y \times \Lambda^\circ / \Lambda^\circ}[z],zd -d_y \phi \wedge \right)\, .
\]

If the semigroup $\mbn A$ is saturated we have the following identification:
\begin{prop}\cite[Proposition 3.20]{RS2}
Let $\mbn A$ be a saturated semigroup. There exists the following $R_{\mbc_z \times \Lambda^\circ}$-linear isomorphism
\[
H^0\left(\Omega^{\bullet +d}_{Y \times \Lambda^\circ / \Lambda^\circ}[z],zd -d_y\phi_A \wedge \right)  \simeq  {\,\,^{\circ\!\!\!\!\!}}{_0}\widehat{M}_{A}\, .
\]
\end{prop}

\begin{prop}\cite[Corollary 2.19]{RS1}\label{cor:PairingPUp}
\begin{enumerate}
\item
There is a non-degenerate flat $(-1)^d$-symmetric pairing
\[
P:\left( {^\circ}\widehat{\mcm}_A\right)_{\mid \mbc^*_z \times \Lambda^\circ} \otimes \iota^*\left({^\circ}\widehat{\mcm}_A\right)_{\mid \mbc^*_z \times \Lambda^\circ}\rightarrow \mco_{\mbc^*_z \times \Lambda^\circ}.
\]
\item
We have that $P(\mclogo,\mclogo)\subset z^d \mco_{\mbc_z\times \Lambda^\circ}$,
and $P$ is non-degenerate on $\mclogo$.
 i.e., it induces a non-degenerate symmetric pairing
\[
[z^{-d}P]:\left[\dfrac{\mclogo}{z\cdot\mclogo}\right]\otimes\left[\dfrac{\mclogo}{z\cdot \mclogo}\right]\rightarrow\mco_{\Lambda^\circ}.
\]
\end{enumerate}
\end{prop}

\section{Construction of the Landau-Ginzburg model}\label{sec:LG}

\subsection{Local Landau-Ginzburg models}\label{ssec:LocalLG}
Let $\mcx$ be a projective toric Deligne-Mumford stack with fan $\mathbf{\Sigma} = (N, \Sigma, \mfa)$.
In this section we explain the construction of a (Zariski local) Landau-Ginzburg model which will serve as a mirror partner for $\mcx$.\\

Recall the sequence \ref{eq:extseq}
\[
0 \lra \mbl \overset{\mft}\lra \mbz^{n} \overset{\mfa}\lra N \lra 0\, .
\]
We apply the functor $\Hom(-,\mbc^*)$ to the sequence above which gives the following sequence of algebraic tori:
\[
1 \lra \Hom(N, \mbc^*)  \lra \mcy' := \Hom( \mbz^{n}, \mbc^*) \overset{\psi_{\mft}}\lra \mct:= \Hom(\mbl,\mbc^*)  \lra 1\, .
\]
The basis $e_1,\ldots , e_n$ equips $\mcy$ with coordinates $w_1,\ldots, w_n$. Consider the map
\begin{align}
W': \mcy' &\lra \mbc_t \times \mct\, , \notag \\
(w_1,\ldots,w_n) &\mapsto (-\sum_{i=1}^n w_i, \psi_{\mft}(\underline{w}))\, . \notag
\end{align}

Usually this map $W'$ is called the Landau-Ginzburg model of the toric orbild $\mcx$. However to ensure a correct limit behavior we need to consider a covering of this model. Consider the inclusion $\mfc: \mbl \hookrightarrow NE^e(X)$, which gives rise to a covering $\mct \leftarrow \mcm:= \Hom(NE^e(X),\mbc^*)$. We get the following cartesian diagram
\[
\xymatrix{\mcy' \ar[d]^{W'} & \mcy \ar[l] \ar[d]^W \\ \mbc_t \times \mct & \mbc_t \times \mcm \ar[l]^{id \times \psi_\mfc}} 
\]

\begin{defn}
Let $\mcx$ be a projective toric orbifold. The mirror Landau-Ginzburg model of $\mcx$ is given by
\begin{align}
W: \mcy &\lra \mbc_t \times \mcm\, . \notag
\end{align}
\end{defn}

We have to compute the Fourier-Laplace transformed Gau\ss-Manin system and the Brieskorn lattice for the map $W$. For this we will construct a map $\mbc^t \times \Lambda^* \leftarrow \mbc_t \times \mct$ such that $W$ becomes the pull-back of the map $\varphi_A$ from Section \ref{sec:LauGKZ} (recall that $A$ is the matrix corresponding to $\mfa$ after the choice of a basis for $N$) under the cocatenated morphisms $ \Lambda^* \leftarrow \mbc_t \times \mct \leftarrow \mbc_t \times \mcm$ .

Here we will identify $\Lambda^*:= \Hom(\mbz^n,\mbc^*)$ with $\Lambda \setminus\{\lambda_1\cdot \ldots \cdot \lambda_n =0\}$.\\

Since we assumed that $N$ is free, we have $\Ext^1(N,\mbz)= 0$ which gives us the following commutative diagram whose vertical maps are isomorphismsm
\[
\xymatrix{0 \ar[r] & \mbl  \ar@{=}[d]\ar[r]^{\mft} & \mbz^{n} \ar[r]^\mfa\ar[d]^\mfk & N \ar[r]\ar@{=}[d] & 0 \\
0 \ar[r] & \mbl  \ar[r] & \mbl \oplus N \ar[r] & N \ar[r] & 0}
\]
where $\mfk = \mfs + \mfa$ with $\mfs: \mbz^{n} \ra \mbl$ and $\mfk^{-1} = \mft +\mfg$ with $\mfg: N \ra \mbz^{n}$.
The maps satisfy the following relations
\begin{enumerate}
\item $\mfa \circ \mft = 0$ and $\mfs \circ \mfg = 0$
\item $\mfa \circ \mfg = id_N$ and $\mfs \circ \mft = id_{\mbl}$
\end{enumerate}
Consider the push-out diagram
\[
\xymatrix{-\sum_{i=1}^n( \chi^{a_i} \otimes \chi^{e_i} \chi^{e_k}) &\mbc[N] \otimes \mbc[\mbz^n] \ar[r] & \mbc[\mbz^n] & \left( -\sum_{i=1}^n \chi^{e_i}\right)\chi^{\mft(l)} \\t \otimes \chi^{e_k}  \ar@{|->}[u]&\mbc[t] \otimes \mbc[\mbz^n] \ar[u] \ar[r] & \mbc[t] \otimes \mbc[\mbl] \ar[u] & t \otimes \chi^l \ar@{|->}[u]\\
& t \otimes \chi^{e_k} \ar[r] & t \otimes \chi^{\mfs(e_k)}}
\]

Here we made the following identifications
\begin{align}
\mbc[N] \otimes \mbc[\mbz^n] \otimes_{\mbc[t] \otimes \mbc[\mbz^n]} \mbc[t] \otimes \mbc[\mbl] &\overset{\simeq}\lra \mbc[N] \otimes \mbc[\mbl] \notag\\
1 \otimes  \chi^{e_i} \otimes 1 \otimes 1 &\mapsto 1 \otimes \chi^{\mfs(e_i)} \notag \\
1 \otimes 1 \otimes t \otimes 1 &\mapsto -\sum_{i=1}^n(\chi^{a_i} \otimes \chi^{\mfs(e_i)}) \notag
\end{align}
and
\begin{align}
\mbc[N] \otimes \mbc[\mbl] &\lra \mbc[\mbz^n] \notag \\
\chi^n \otimes \chi^l &\mapsto \chi^{\mfg(n)}\chi^{\mft(l)}\, . \notag
\end{align}
This gives a cartesian diagram
\[
\xymatrix{Y \times \Lambda^* \ar[d]_{\varphi_A} & \mcy' \ar[d]^{W'} \ar[l] \\ \mbc_t \times \Lambda^* &  \mbc_t \times \mct \ar[l]^{id \times \psi_{\mfs}} }
\]
We denote by $\psi$ the concatenation $\psi_\mfs \circ \psi_\mfc$ and get the cartesian diagram
\[
\xymatrix{Y \times \Lambda^* \ar[d]_{\varphi_A} & \mcy \ar[d]^{W} \ar[l] \\ \mbc_t \times \Lambda^* &  \mbc_t \times \mcm \ar[l]^{id \times \psi} }
\]

Choose some $\mbz$-basis $p_1,\ldots,p_{r+e}$ in  $Pic^e(\mcx)$ which satisifes the following conditions:
\begin{enumerate} 
\item $p_1,\ldots , p_r \in \theta(\mck)\subset \mck^e$,
\item $p_{r+i} = [D_{m+i}]$ for $i=1,\ldots,e$,
\item $\rho \in Cone(p_1,\ldots, p_{r+e})$.
\end{enumerate} 
 
The dual $\mbz$-basis of $(p_a)_{a=1,\ldots,r+e}$ equips $\mcm_\mcx = Hom(NE^e(X),\mbc^*)$ with coordinates $\chi_1,\ldots, \chi_{r+e}$. \\

Using the coordinates $\chi_a$ the map $\psi$ is given by
\begin{align}
\psi = \psi_\mfn:  \mcm_\mcx &\lra \Lambda^\ast\, , \notag \\
(\chi_1,\ldots,\chi_r) &\mapsto (\underline{\chi}^{\underline{n}_1},\ldots, \underline{\chi}^{\underline{n}_{m+e}}) \notag
\end{align}
where $\mfn := \mfc \circ \mfs$ with
\[
\mfn(e_i)  = \sum_{a=1}^r n_{ai}p_a\, .
\]

After the  choice of the splitting above the Landau-Ginzburg model for $\mcx$ is given by
\begin{align}
W: \mcy \simeq Y \times \mcm_\mcx &\lra \mbc_t \times \mcm_\mcx\, , \label{eq:LGLaurent} \\
(\underline{y},\underline{\chi}) &\mapsto (-\sum_{i=1}^{n} \underline{\chi}^{n_i} \underline{y}^{a_i}, \underline{\chi})\, . \notag
\end{align}

Since  $\mbl$ is a full sublattice in $NE^e(X)$ the map $\mfc$ becomes an isomorphism after tensoring with $\mbq$. Let
\[
\mfr: NE^e(X) \otimes \mbq \lra \mbl \otimes \mbq
\]
be its inverse. We denote by $\mfm$ the concatenation $\mft \circ \mfr$. With respect to the basis $e_1,\ldots, e_n$ and the dual basis of $p_1,\ldots, p_{r+e}$ the map $\mfm$ is given by the matrix $M=(m_{ia})$.

It follows that
\begin{equation}\label{eq:charmia}
[D_i] = \sum_{a=1}^{r+e} m_{ia}p_a\, .
\end{equation}

Now we want to compute the Fourier-Laplace transformed Gau\ss-Manin system of the map $W$. We do this by computing an inverse image of the FL-transformed Gau\ss-Manin system of $\varphi_A$.\\

\begin{prop}\label{prop:RSDmodds}
Let $\widetilde{\psi} = (id_z,\psi): \mbc_z \times \mcm \ra \mbc_z \times \Lambda^\ast$.
\begin{enumerate}
\item The inverse image $\qma:= \widetilde{\psi}^+ ({^*}\widehat{\mcm}_A)$ is isomorphic to the quotient of $\mcd_{\mbc_z \times \mcm_\mcx} / \mci$, where $\mci$ is the left ideal generated by
\[
\widetilde{\Box}_{\underline{l}}:=\!\!\!\prod_{a:p_a(\underline{l})>0}\!\!\!\! {\chi_a}^{p_a(\underline{l})}\! \prod_{i:l_i<0}
\prod_{\nu=0}^{-l_i-1} (\sum_{a=1}^{r+e} m_{ia} z\chi_a\partial_{\chi_a} -\nu z) -
\!\!\!\prod_{a:p_a(\underline{l})<0}\!\!\!\! {\chi_a}^{-p_a(\underline{l})}\! \prod_{i:l_i>0}
\prod_{\nu=0}^{l_i-1} (\sum_{a=1}^{r+e} m_{ia} z\chi_a\partial_{\chi_a} -\nu z)
\]
for any $\underline{l} \in \mbl$ and by the single operator
\[
\check{E}:=  z^2 \p_z + \sum_{a=1}^{r+e} \sum_{i=1}^{m+e}m_{ia} z \chi_a \p_{\chi_a}\, .
\]
\item There is an isomorphism of $\mcd_{\mbc_z \times \mcs}$-modules
\[
\qma  \simeq  \FL_{\mcm_{\mcx}}^{loc}(\mch^0 W_+ \mco_{Y\times \mcm_\mcx})\, .
\]
\end{enumerate}
\end{prop}
\begin{proof}
We first choose bases $w_1,\ldots, w_{r+e}$ of $\mbl$ and $v_1,\ldots, v_d$ of $N$ and denote the dual bases by $w^*_1,\ldots, w^*_{r+e}$ and $v^*_1,\ldots, v^*_d$, respectively. This gives rise to coordinates $\tau_1,\ldots, \tau_{r+e}$ on $\mct = \Hom(\mbl,\mbc^*)$ and and $h_1,\ldots, h_d$ on $H= \Hom(N,\mbc^*)$. We set $a_i = \mfa(e_i)$, $s_i = \mfa(e_i)$ for $i=1,\ldots,n$ and $t_a = \mft(w_a)$ resp. $g_{j} = \mfg(v_j)$ for $a=1,\ldots r+e$ resp. $j = 1,\ldots ,d$.\\

We will first compute the inverse image under the map $\psi_\mfs$, which was induced by the linear morphism $\mfs: \mbz^{n} \ra \mbl$. We factor $\mfs$ in the following way:
\[
\mfs: \mbz^{n} \overset{(\mfs,\mfa)}\lra \mbl \oplus N \overset{p_1}\lra \mbl
\]
where $p_1$ is the projection to the first factor.  Hence we get a factorization of $\psi_\mfs$ given by

\begin{align}
\delta: \mct &\lra \mct \times H\, , \notag \\
(\tau_1, \ldots , \tau_r) &\mapsto (\tau_1, \ldots , \tau_r,1,\ldots,1)\, , \notag \\[0.7em]
i_1: \mct \times H &\lra \Lambda^*\, , \notag\\[-0.2em]
(\tau_1,\ldots , \tau_r, h_1,\ldots, h_d) &\mapsto (\underline{\tau}^{s_1}\cdot \underline{h}^{a_1},\ldots,\underline{\tau}^{s_{m+e}}\cdot \underline{h}^{a_{m+e}} ) \notag
\end{align}
where $\underline{\tau}^{s_1} := \prod_{b=1}^{r+e}\tau_b^{s_{b1}}$ etc. and the inverse of $i_1$ is given by
\begin{align}
i_1^{-1}: \Lambda^* &\lra \mct \times H\, , \notag \\
(\lambda_1,\ldots, \lambda_{m+e}) &\mapsto (\underline{\lambda}^{\underline{t}_1},\ldots, \underline{\lambda}^{\underline{t}_r}, \underline{\lambda}^{\underline{g}_1},\ldots, \underline{\lambda}^{\underline{g}_d })\, . \notag
\end{align}

Notice that the ideal $\mci$ in Definition \ref{def:FLGKZ} is generated by
\begin{align}
\left(\prod_{i: l_i >0} \lambda_i^{l_i}\right) \cdot \hat{\Box}_{\underline{l}} &= \prod_i \lambda_i^{l_i} \cdot \prod_{i: l_i <0} (\lambda_i)^{-l_i}(z   \p_{\lambda_i})^{-l_i} - \prod_{i : l_i >0} (\lambda_i)^{l_i}(z  \p_{\lambda_i})^{l_i}\quad \text{for}\; \underline{l} \in \mbl \, ,\notag \\
&=\prod_i \lambda_i^{l_i} \cdot \prod_{i: l_i <0}\prod_{\nu =0}^{-l_i-1} (z  \lambda_i \p_{\lambda_i}-\nu z)^{-l_i} - \prod_{i : l_i >0}\prod_{\nu=0}^{l_i-1} (z \lambda_i \p_{\lambda_i}-\nu z)^{l_i}\quad \text{for}\; \underline{l} \in \mbl \, ,\notag \\
\hat{E} &:= z^2 \p_z + \sum_{i=1}^{m+e} z \lambda_i \p_{\lambda_i}\, , \notag \\
\hat{E}_k &:= \sum_{i=1}^{m+e} a_{ki} z \lambda_i \p_{\lambda_i} \, . \notag
\end{align}

We have the following transformation rules for the coordinate change $i_1$:
\[
\lambda_i \p_{\lambda_i} \mapsto \sum_{b=1}^{r+e} t_{ib} \tau_b \p_{\tau_{b}} + \sum_{k=1}^d g_{ik} h_k \p_{h_k} 
\]
where $t_{ib} =D_i(\mft(w_b))$.\\

Since $i_1$ is an isomorphism, we have that $i_1^+({^*}\widehat{\mcm}_A) $ is isomorphic to $\mcd_{\mbc_z \times \Lambda^*}/ \mci'$ where the left ideal $\mci'$ is generated by 
\[
\Box_{\underline{l}}':=\prod_{b=1}^{r+e} \tau_b^{w_b^*(\underline{l})} \prod_{i:l_i<0}
\prod_{\nu=0}^{-l_i-1} (\sum_{b=1}^{r+e} D_i(\mft(w_b)) z\tau_b\partial_{\tau_b} -\nu z) -
 \prod_{i:l_i>0}
\prod_{\nu=0}^{l_i-1} (\sum_{b=1}^{r+e} D_i(\mft(w_b)) z\tau_b\p_{\tau_b} -\nu z)
\]
for any $\underline{l} \in \mbl$ and by the Euler operators
\begin{align}
& E':= z \p_z + \sum_{b=1}^{r+e} \left(\sum_{i=1}^{m+e}D_i(\mft(w_b))\right) \tau_b \p_{\tau_b}\, , \notag \\
&E'_k :=h_k \p_{h_k} \qquad \text{for} \; k=1,\ldots ,d\, . \notag
\end{align}
Here, we have used the formulas
\[
\prod_{i=1}^{m+e} {\lambda_i}^{l_i} = \prod_{i=1}^{m+e} \underline{\tau}^{l_i \cdot\underline{s}_i} \cdot \underline{h}^{l_i \cdot \underline{a}_i} = \prod_{b=1}^{r+e}\tau_b^{\sum_i l_i s_{bi}} \prod_{k=1}^d h_k^{\sum_i l_i a_{ki}} = \prod_{b=1}^{r+e}\tau_b^{\sum_i l_i s_{bi}} 
\]
and
\[
\sum_i l_i s_{ib} = \sum_i  l_i w_b^*(\mfs(e_i)) = w_b^*(\mfs(\mft(\underline{l})) = w_b^*(\underline{l}).
\]

It is now easy to see that the inverse image $\psi_\mfs^+({^*}\widehat{\mcm}_A) \simeq \delta^+ i_1^+({^*}\widehat{\mcm}_A)$ is isomorphic to $\mcd_{\mbc_z \times \mct} / \mci''$ where the left ideal $\mci''$ is generated by $\Box_{\underline{l}}'$ and $E'$.

We will now compute the inverse image under $\psi_\mfc$. Denote by $q_1,\ldots, q_{r+e}$ the basis in $NE^e(X)$ dual to $p_1,\ldots, p_{r+e} \subset Pic^e(X)$.  With respect to the bases $w_1,\ldots,w_{r+e}$ resp. $q_1,\ldots, q_{r+e}$ the $\mbz$-linear map $\mfc$ is given by a matrix $C =(c_{ab})$, i.e. $\mfc(w_b) = \sum_{a=1}^{r+e} c_{ab}q_a$.  
We factorize this matrix to obtain
\[
C = C_1 \cdot D \cdot C_2
\]
with $C_1=(c^1_{ab}),C_2=(c^2_{ba}) \in GL(r,\mbz)$ and $D=diag(d_1,\ldots,d_{r+e})$ a diagonal matrix with strictly positive integer entries.

The factorization of $C$ gives also a factorization of $R = C^{-1}$, i.e.
\[
R = R_2 \cdot D^{-1} \cdot R_1
\]
with $R_i=(r^i_{ab}) = C_i^{-1} \in Gl(r,\mbz)$. We define new bases
\[
w'_h = \sum_{b=1}^{r+e} r^2_{bh} w_b \qquad \text{and} \qquad q'_h = \sum_{a=1}^{r+e} c^1_{ah} q_a\, .
\]
With respect to these bases $\mfc$ is diagonal, i.e. $\mfc(w'_h) = d_h \cdot q'_h$.\\

The choice of these bases gives rise to coordinate changes on $\mct$ and $\mcm_\mcx$:
\[
\tau'_h = \prod_{b=1}^{r+e} \tau_b^{r^2_{bh}} \qquad \text{and} \qquad \chi'_h = \prod_{a=1}^{r+e} \chi_a^{c^1_{ah}}
\]
with inverses
\[
\tau_b = \prod_{h=1}^{r+e} (\tau'_h)^{c^2_{hb}} \qquad \text{and}\qquad  \chi_a = \prod_{h=1}^{r+e} (\chi_h')^{r^1_{ha}}\, .
\]

Hence we get a factorization of $\psi_\mfc = \psi_2 \circ \kappa \circ \psi_1$,  where the maps are given by
\begin{align}
\psi_1: \mct &\lra \mct\, , \notag \\
(\tau'_1, \ldots , \tau'_{r+e}) &\mapsto (\tau_1 = \prod_{h=1}^{r+e} (\tau'_h)^{c_{h1}^1}, \ldots , \tau_{r+e} = \prod_{h=1}^{r+e} (\tau'_h)^{c_{h,r+e}^1})\, , \notag \\
\kappa: \mcm_\mcx &\lra \mct\, , \notag \\
(\chi'_1, \ldots , \chi'_{r+e}) &\mapsto (\tau'_1= (\chi'_1)^{d_1}, \ldots , \tau'_{r+e} = (\chi'_{r+e})^{d_{r+e}})\, , \notag \\
\psi_2: \mcm_\mcx &\lra \mcm_\mcx\, , \notag\\
(\chi_1, \ldots , \chi_{r+e}) &\mapsto (\chi'_1 = \prod_{a=1}^{r+e} \chi_a^{c_{a1}^2}, \ldots , \chi'_{r+e} = \prod_{a=1}^{r+e} (\chi_a)^{c_{a,r+e}^2})\, . \notag
\end{align}

Notice that we have the following transformation rule:
\[
\tau_b \p_{\tau_b} \mapsto \sum_{h=1}^{r+e}r^2_{bh}\tau_h'\p_{\tau'_h}\, .
\]
Since $\psi_1$ is an isomorphism $\psi_1^+ \mcd_{\mbc_z \times \mct}/\mci''$ is isomorphic to $\mcd_{\mbc_z \times \mct} / \mcj''$ where $\mcj''$ is generated by
\[
\Box_{\underline{l}}':=\prod_{h=1}^{r+e} (\tau'_h)^{(w'_h)^*(\underline{l})} \prod_{i:l_i<0}
\prod_{\nu=0}^{-l_i-1} (\sum_{h=1}^{r+e} D_i(\mft(w'_h)) z\tau'_h\partial_{\tau'_h} -\nu z) -
 \prod_{i:l_i>0}
\prod_{\nu=0}^{l_i-1} (\sum_{h=1}^{r+e} D_i(\mft(w'_h)) z\tau'_h\p_{\tau'_h} -\nu z)
\]
where $\underline{l} \in \mbl$ and
\[
z^2 \p_z + \sum_{h=1}^{r+e} \left(\sum_{i=1}^{m+e} D_i(\mft(w'_h))\right)z \tau'_h \p_{\tau'_h}\, .
\]
Here we have used that
\[
\prod_{b=1}^{r+e} \tau_b^{w_b^*(\underline{l})} = \prod_{b=1}^{r+e} \left(\prod_{h=1}^{r+e} (\tau'_h)^{c^2_{hb}} \right)^{w_b^*(\underline{l})} = \prod_{h=1}^{r+e}(\tau'_h)^{\sum_{b=1}^{r+e} c^2_{hb} w_b^*(\underline{l})} = \prod_{h=1}^{r+e}(\tau'_h)^{(w'_h)^*(\underline{l})}
\]
and
\begin{align}
\sum_{b=1}^{r+e} D_i(\mft(w_b)) z \tau_b \p_{\tau_b} &= \sum_{b=1}^{r+e} D_i(\mft(w_b)) z \left(\sum_{h=1}^{r+e}r^2_{bh}\tau_h'\p_{\tau'_h}\right)\notag \\ 
&= \sum_{h=1}^{r+e} D_i(\mft( \sum_{b=1}^{r+e} r^2_{bh} w_b)) z \tau'_h \p_{\tau'_h} \notag \\
&= \sum_{h=1}^{r+e} D_i(\mft(w'_h)) z \tau'_h \p_{\tau'_h}\, . \notag
\end{align}

In order to compute $\kappa^+ \psi_1^+ \mcd_{\mbc_z \times \mct} / \mci '' \simeq \kappa^+ \mcd_{\mbc_z \times \mct}/ \mcj$, we first notice that
\[
\kappa^+ \mcd_{\mbc_z \times \mct}/ \mcj  \simeq \mco_{\mcm_\mcx} \otimes \kappa^{-1}(\mcd_{\mbc_z \times \mct}/ \mcj)
\]
where the operators $\chi'_h$ resp. $\chi'_h \p_{\chi'_h}$ act by
\[
(\chi'_h)^{d_n} (f \otimes P) = f \otimes \tau'_h P 
\]
resp.
\[
\chi'_h \p_{\chi'_h} ( f \otimes P) = \chi'_h \p_{\chi'_h}(f)\otimes P + f \otimes (d_h\tau'_h \p_{\tau'_h}) P\, .
\]
An easy computation shows that $\kappa^+ \mcd_{\mbc_z \times \tau}/ \mcj''$ is isomorphic to the quotient $\mcd_{\mbc_z \times \mct} / \mcj'$ where the left ideal $\mcj'$ is generated by
\[
\prod_{h=1}^{r+e} (\chi'_h)^{\mfc^*(q'_h)^*(\underline{l})} \prod_{i:l_i<0}
\prod_{\nu=0}^{-l_i-1} (\sum_{h=1}^{r} D_i(\mft\circ \mfr(q'_h)) z\chi'_h\partial_{\chi'_h} -\nu z) -
 \prod_{i:l_i>0}
\prod_{\nu=0}^{l_i-1} (\sum_{h=1}^{r} D_i(\mft\circ \mfr(q'_h)) z\chi'_h\p_{\chi'_h} -\nu z)
\]
for any $\underline{l} \in \mbl$ and by the single operator
\[
z^2 \p_z + \sum_{h=1}^{r+e} \left(\sum_{i=1}^{m+e} D_i(\mft\circ \mfr(q'_h))\right)z \chi'_h \p_{\chi'_h}\, .
\]
Here we used $\mfc(w'_h)=d_h \cdot p'_h$, i.e. $\mfc^*((q'_h)^*) =d_h\cdot  (w'_h)^*$.

The final step consists in computing $\psi_2^+ \kappa^+ \psi_1^+ \mcd_{\mbc_z \times \mct} / \mci'$ which is completely parallel to the computation of the inverse image under $\psi_1$. Therefore the first claim follows.\\

For the second claim consider the cartesian diagram
\[
\xymatrix{Y \times \mcm_\mcx \ar[rr] \ar[d]_W & & Y \times \Lambda^* \ar[d]^{\varphi_A} \\ \mbc_{\lambda_0} \times \mcm_\mcx \ar[rr]^{(id \times \psi)}& &\mbc_{\lambda_0} \times \Lambda^*}
\]
We have the following isomorphisms
\begin{align}
\qma  &\simeq \widetilde{\psi}^+ ({^*}\widehat{\mcm}_A)\notag \\ 
&\simeq \widetilde{\psi}^+\FL^{loc}_{\Lambda}   \mch^0(\varphi_{A,+} \mco_{Y \times \Lambda^*}) \notag \\ 
&\simeq   \FL_{\mcm_{\mcx}}^{loc} (id_{\mbc_{\lambda_0}}\times \psi)^+ \mch^0(\varphi_{A,+} \mco_{Y \times \Lambda}) \notag \\
&\simeq   \FL_{\mcm_{\mcx}}^{loc}(\mch^0 W_+ \mco_{Y\times \mcm_\mcx}) \notag
\end{align}

where the third isomorphism follows from the compatibility of the localized Fourier-Laplace transform with base change and the fourth isomorphism is base change with respect to the diagram above.\\
\end{proof}

Write $\mcm_\mcx^\circ  := \psi^{-1}(\Lambda^\circ) = \{ (\chi_1, \ldots , \chi_{r+e}) \in \mcm_\mcx\! \mid W =\! -\sum_{i=1}^{m+e} \underline{\chi}^{\underline{n}_i} \underline{y}^{\underline{a}_i} \; \text{is Newton non-degenerate} \}$. 

We have the following statement for the Brieskorn-lattice:
\begin{prop}\label{prop:PropzExt}$ $\\
\begin{enumerate}
\item
The inverse image (in the category of $\mco$-modules) 
\begin{equation}\label{eq:invImLat}
\qmclogo:= \widetilde{\psi}^*(\mclogo)  = \mco_{\mbc_z \times \mcm_\mcx^\circ} \otimes \widetilde{\psi}^{-1}(\mclogo)
\end{equation}
carries a natural structure of an $\mcr_{\mbc_z \times \mcm_\mcx^\circ}$-module.
It is isomorphic to the quotient $\mcr_{\mbc_z \times \mcm_\mcx^\circ}/\mci_0$, where $\mci_0$ is the left ideal generated by $(\widetilde{\Box}_{\underline{l}})_{\underline{l} \in \mbl}$ and $\check{E}$.
\item There exists the following $R_{\mbc_z \times \mcm_\mcx^\circ}$-linear isomorphism
\[
H^0\left(\Omega^{\bullet +r+e}_{Y \times \mcm_\mcx^\circ / \mcm_\mcx^\circ}[z],zd -d_yF \wedge \right)  \simeq  {_0\!}{^\circ\!}QM_{A}
\]
where ${_0\!}{^\circ\!}QM_{A} = \Gamma(\mcm_\mcx^\circ,\qmclogo)$.
\item
There is a non-degenerate flat $(-1)^d$-symmetric pairing
$$
P:\left(\qma\right)_{\mid \mbc_z^* \times \mcm_\mcx^\circ} \otimes \iota^*\left(\qma\right)_{\mid \mbc^*_z \times \mcm_\mcx^\circ} \rightarrow \mco_{\mbc^*_z\times \mcm_\mcx^\circ}.
$$
\item
$P(\qmclogo ,\qmclogo )\subset z^d \mco_{\mbc_z\times \mcm_\mcx^\circ}$,
and $P$ is non-degenerate on $\qmclogo$.
\end{enumerate}
\end{prop}
\begin{proof}
First notice that the map $\widetilde{\psi}$ factorizes as $(id \times \psi_\mfs) \circ ( id \times \psi_\mfc)$. The map $(id \times \psi_\mfs)$ is non-characteristic with respect to $\mclog$ since the singular locus  of $\mclog$ is contained in $(\{0,\infty\} \times \Lambda^\circ)$ and the map $(id \times \psi_\mfc)$ is non-characteristic with respect to any coherent $\mcd_{\mbc_z \times \mct}$-module since $(id \times \psi_\mfc)$ is smooth. Hence the inverse image is nothing but the inverse image in the category of meromorphic connections. The inverse image of the lattice $\mclogo$ is then simply given by the formula \eqref{eq:invImLat}.

The second point follows by base change and the fact that $\widetilde{\psi}^*= (id \times \psi_\mfc)^* \circ (id \times \psi_\mfs)^*$ is exact. The third and fourth point  follow from Proposition \ref{cor:PairingPUp}.
\end{proof}

\begin{lem}\label{lem:facBox}$ $\\[-1.5em]
\begin{enumerate}
\item The $\mcd_{\mbc_z \times \mcm_\mcx}$-module $\qma$ is isomorphic to the quotient $\mcd_{\mbc_z \times \mct}/ \mci $ where $\mcj$ is the left ideal generated by $\check{E}$ and
\begin{align}
\Box^X_{\underline{l}} := &\prod_{a=1 \atop p_a(\underline{l}) > 0}^r \chi_a^{p_a(\underline{l})} \prod_{i=1 \atop l_i < 0}^m \prod_{\nu =0}^{-l_i-1}(\msd_i - \nu z) \prod_{i=m+1 \atop l_i < 0}^{m+e} \msd_i^{-l_i} 
- \prod_{a=1 \atop p_a(\underline{l})}^r \chi_a^{-p_a(\underline{l})} \prod_{i=1 \atop l_i > 0}^m \prod_{v=0}^{l_i-1} (\msd_i - \nu z) \prod_{i=m+1 \atop l_i >0}^{m+e}\msd_i^{l_i} \notag
\end{align}
where
\[
\msd_i = \begin{cases}\sum_{a=1}^{r+e} m_{ia}z \chi_a \p_{\chi_a} & \text{for}\;  i=1,\ldots,m\\ z\p_{\chi_{i-m+r}} & \text{for}\; i = m+1,\ldots , m+e\, . \end{cases}
\]
\item The $\mcr_{\mbc_z \times \mcm_\mcx^\circ}$-module $\qmclogo$ is isomorphic to the quotient $\mcr_{\mbc_z \times \mcm_\mcx^\circ} /\mcj_0$ where $\mcj_0$ is the left ideal generated by $\Box^X_{\underline{l}}$ and $\check{E}$.

\end{enumerate}
\end{lem}
\begin{proof}
Notice that we have the following identifications
\begin{enumerate}
\item  $p_{r+i}(\underline{l}) =[D_{m+i}](\underline{l}) = l_{m+i}$ for $i=1,\ldots,e$,
\item $m_{m+i,a} = \delta_{r+i,a}$
\end{enumerate}
where the second point follows from formula \ref{eq:charmia}. We can therefore write
\begin{align}
\widetilde{\Box}_{\underline{l}}&=\!\!\!\prod_{a:p_a(\underline{l})>0}\!\!\!\! {\chi_a}^{p_a(\underline{l})}\! \prod_{i:l_i<0}
\prod_{\nu=0}^{-l_i-1} (\sum_{a=1}^{r+e} m_{ia} z\chi_a\partial_{\chi_a} -\nu z) -
\!\!\!\prod_{a:p_a(\underline{l})<0}\!\!\!\! {\chi_a}^{-p_a(\underline{l})}\! \prod_{i:l_i>0}
\prod_{\nu=0}^{l_i-1} (\sum_{a=1}^{r+e} m_{ia} z\chi_a\partial_{\chi_a} -\nu z) \notag \\
&= \prod_{a=1 \atop p_a(\underline{l})>0}^r \chi_a^{p_a(\underline{l})} \prod_{i=1 \atop l_{m+i} > 0}^e \chi_{r+i}^{l_{m+i}} \prod_{i=1 \atop l_i < 0}^{m} \prod_{\nu=0}^{-l_i-1}(\msd_i - \nu z)\prod_{i=1 \atop l_{m+i} < 0}^e \chi_{r+i}^{-l_{m+i}}\prod_{i=m+1 \atop l_{i} < 0}^{m+e} \msd_i^{-l_i} \notag \\
&- \prod_{a=1 \atop p_a(\underline{l})<0}^r \chi_a^{-p_a(\underline{l})} \prod_{i=1 \atop l_{m+i} < 0}^e \chi_{r+i}^{-l_{m+i}} \prod_{i=1 \atop l_i > 0}^{m} \prod_{\nu=0}^{l_i-1}(\msd_i - \nu z)\prod_{i=1 \atop l_{m+i} > 0}^e \chi_{r+i}^{l_{m+i}}\prod_{i=m+1 \atop l_{i} > 0}^{m+e} \msd_i^{l_i} \notag\\
&= \prod_{i=1}^e \chi_{r+i}^{|l_{m+i}|}\cdot\Box^X_{\underline{l}}\, . \notag
\end{align}
Since the $\chi_a$ are invertible on $\mbc_z \times \mcm_\mcx$ this shows the first and second point.
\end{proof}

\subsection{Logarithmic extension}
Let $Y$ be a smooth variety and $D$ be a reduced normal-crossing divisor in $Y$. Denote by $\mcr_{\mbc_z \times Y}(log D)$ the subsheaf of $\mcr_{\mbc_z \times 
Y}$ generated by $\mco_{\mbc_z \times Y}$, $z^2 \p_z$ and $z\cdot p^{-1}Der(log D)$, where $p : \mbc_z \times Y \ra Y$ is the canonical projection and $Der(logD)$ the sheaf of logarithmic vector fields along $D$.\\

Recall the definition of the base space $\mcm_\mcx = \Hom(NE^e(X),\mbc^*) \simeq (\mbc^*)^{r+e}$ of the Landau-Ginzburg model from section \ref{ssec:LocalLG}.. The choice of a basis $p_1,\ldots,p_{r+e}$ determines a partial compactification $\overline{\mcm}_\mcx \simeq \mbc^{r+e}$.  Let $D_\mcx \subset \overline{\mcm}_\mcx$ be the normal crossing divisor given by $\chi_1 \cdots \chi_{r} = 0$. 

Denote by $\Delta := \mcm_\mcx \setminus \mcm_\mcx^\circ$ and let $\overline{\Delta}$ be the closure of $\Delta$ in $\overline{\mcm}_\mcx$. Define $\overline{\mcm}_\mcx^\circ := \overline{\mcm}_\mcx \setminus \overline{\Delta}$. We denote by $p_\mcx$ the point with coordinates $\chi_1 = \ldots = \chi_{r+e} = 0$.
\begin{lem}\label{lem:lvminM}
The point $p_\mcx$ is contained in $\overline{\mcm}_\mcx^\circ$.
\end{lem}
\begin{proof}
This follows with an easy adaption of the proof in \cite[Appendix 6.1]{Ir2}.
\end{proof}

\begin{defn}
Let ${_0}\qma^{log,\mcx}$ be the quotient $\mcr_{\mbc_z \times \overline{\mcm}_\mcx^\circ}(log D_X)/\mci_X$ where $\mci_X$ is the ideal generated by  $(\Box^X_{\underline{l}})_{\underline{l} \in \mbl}$ and $\check{E}$.
\end{defn}

\begin{thm}\label{thm:latcoh}
There is a Zariski open subset $\mcu_\mcx\subset \overline{\mcm}_\mcx^\circ$ containing the point $p_\mcx$ such that $\qmclogo^{log,\mcx} := {{_0}\qma^{log,\mcx}}_{\mid \mcu_\mcx}$ is $\mco_{\mbc_z \times \mcu_\mcx}$-coherent.
\end{thm}
\begin{proof}
Let $\mcr'(log D_\mcx)$ be the sheaf associated to the ring 
\[
\mbc[z,\chi_1, \ldots , \chi_r]\langle z\chi_1 \p_{\chi_1}, \ldots , z \chi_r \p_{\chi_r}, z\p_{\chi_{r+1}}, \ldots, z\p_{\chi_{r+e}}\rangle.
\]
 Notice that $\qmclogo^{log,\mcx}$ carries a natural $\mcr'(log D_\mcx)$-module structure. We denote the corresponding $\mcr'(log D_\mcx)$-module by $For_{z^2 \p_z}( \qmclogo^{log,\mcx})$.  Notice that it is enough to prove the coherence at the point $p_\mcx$ for $For_{z^2 \p_z}(\qmclogo^{log,\mcx})$ since it is isomorphic to $\qmclogo^{log,\mcx}$ as a $\mco_{\mbc_z \times \mcu_X}$-module.\\

Because of the operator $\check{E} = z^2 \p_z + \sum_{a=1}^{r+e} \sum_{i=1}^{m+e}m_{ia} z \chi_a \p_{\chi_a}$ the module $For_{z^2 \p_z}(\qmclogo^{log,\mcx})$ is isomorphic to $\mcr'(log D_\mcx) / (\Box^X_{\underline{l}})_{\underline{l} \in \mbl}$. Now consider on $\mcr'(log D_\mcx)$ the natural filtration $F_\bullet$ given by the orders of operators, i.e. the filtration $F_\bullet \mcr'(log D_\mcx)$ is given on the level of global sections by
\[
F_k \mbc[z, \chi_1, \ldots , \chi_r]\langle z\chi_1 \p_{\chi_1}, \ldots , z \chi_r \p_{\chi_r}\rangle := \left\lbrace P \mid P = \sum_{|\underline{s}| \leq k} g_{\underline{s}}(z,\underline{\chi})(z \chi_1\p_{\chi_1})^{s_1}\cdot \ldots \cdot (z \chi_r\p_{\chi_r})^{s_r}\right\rbrace\, .
\]
This filtration induces a filtration $F_{\bullet}$ on $For_{z^2 \p_z}(\qmclogo^{log,\mcx})$ which is good, i.e.
\[
F_k \mcr'(log D_\mcx) \cdot F_l For_{z^2 \p_z}(\qmclogo^{log,\mcx}) = F_{k+l}For_{z^2 \p_z}(\qmclogo^{log,\mcx})\, .
\]
We have a natural identification
\[
gr_{\bullet}^F(\mcr'(log D_\mcx)) = \pi_* \mco_{\mbc_z \times T^* \mcu_\mcx(log D_\mcx)}
\]
where $T^* \mcu_\mcx(log D_\mcx)$ is the total space of the vector bundle associated to the locally free sheaf $\Omega^1_{\mcu_\mcx}(log D_\mcx)$ and $\pi: \mbc_z \times T^* \mcu_\mcx(log D_\mcx) \ra \mbc_z \times \mcu_\mcx$ is the projection. The symbols of all operators $\Box^X_{\underline{l}}$ for $\underline{l} \in \mbl$ cut out a subvariety $\mbc_z \times S$ of $\mbc_z \times T^*\mcu_\mcx(log D_\mcx)$.

It will be sufficient to show that the fiber over $\underline{\chi} = 0$ of $S \ra \mcu_\mcx$ is quasi-finite since this implies that $S \ra \mcu_\mcx$ is quasi-finite in a Zariski open neighborhood of  $\underline{\chi} = 0$. Since $S$ is homogeneous this shows that $S$ is equal to the zero section of $T^*\mcu_X(log D_\mcx)$ over this neighborhood. Adapting a well-known argument from the theory of $\mcd$-modules (see, e.g. \cite{Ph1}) we see that the filtration $F_\bullet$ will become eventually stationary and we conclude by the fact that all $F_k For_{z^2 \p_z}(\qmclogo^{log,\mcx})$ are stationary in this neighborhood.\\

Therefore, it remains now to prove that that the fiber over $z = \underline{\chi} = 0$ of $S \ra \mcu_\mcx$ is quasi-finite.\\

First notice that in the limit $z= \underline{\chi} = 0$ the operators
\[
\msd_i = \begin{cases}\sum_{a=1}^{r+e} m_{ia}z \chi_a \p_{\chi_a} & \text{for}\;  i=1,\ldots,m\\ z\p_{\chi_{i-m+r}} & \text{for}\; i = m+1,\ldots , m+e \end{cases}
\]
in $\mcr'(log D_\mcx)$ degenerate to
\begin{equation}\label{eq:degDi}
\mathbf{D}_i = \begin{cases}\sum_{a=1}^{r} m_{ia}z \chi_a \p_{\chi_a} & \text{for}\;  i=1,\ldots,m\\ z\p_{\chi_{i-m+r}} & \text{for}\; i = m+1,\ldots , m+e\, . \end{cases}
\end{equation}

Since the fan $\Sigma$ is simplicial, we have for each $a_j$, $j=m+1,\ldots, m+e$ a cone relation $\underline{l}_{C_j}$:
\begin{equation}\label{eq:conerelinpro}
\sum_{i=1}^m l_i a_i - l_j a_j  = 0
\end{equation}
with $l_i , l_j \in \mbz_{\geq 0}$.
Because of $p_1, \ldots , p_r \in \theta(\mck) \subset \theta(Pic(X))$ and the definition of the map $\Theta$ one easily sees that 
\begin{equation}\label{eq:vanConerel}
p_a(\underline{l}) = 0
\end{equation}
for all cone relations $\underline{l}$. Hence the corresponding Box operator is
\[
\Box^X_{\underline{l}_{C_j}} = (\msd_j)^{l_j} -  \prod_{i=1}^m \prod_{\nu = 0}^{l_i-1}(\msd_i - \nu z)\, .
\]
Because $\deg(a_i) = 1$ for $i=1,\ldots,m$ and $\deg(a_j) \leq 1$ by Lemma \ref{lem:degai} we get from \eqref{eq:conerelinpro} the inequality $l_j \geq \sum_{i=1}^m l_i$. Hence the symbol of $\Box^X_{\underline{l}_{C_j}}$ is
\begin{equation}\label{eq:SigmaCone}
\sigma(\Box^X_{\underline{l}_{C_j}}) = \begin{cases}\sigma(\msd_j)^{l_j} & \text{if}\; \deg(a_j) < 1 \\ \sigma(\msd_j)^{l_j} - \prod_{i=1}^m \sigma(\msd_i)^{l_i} & \text{if} \deg(a_j) = 1 \end{cases}
\end{equation}
which degenerates to
\begin{equation}\label{eq:SigmaConedeg}
\sigma(\Box^X_{\underline{l}_{C_j}})_{z=\underline{\chi}=0} = \begin{cases}\sigma(\mathbf{D}_j)^{l_j} & \text{if}\; \deg(a_j) < 1 \\ \sigma(\mathbf{D}_j)^{l_j} - \prod_{i=1}^m \sigma(\mathbf{D}_i)^{l_i} & \text{if} \deg(a_j) = 1\, . \end{cases}
\end{equation}

Now suppose that $\{a_i \mid i \in I\}$ for $I \subset \{1,\ldots ,m\}$  is a primitive collection. Denote by $\underline{l}_I \in \mbl$ a primitive relation
\[
\sum_{i\in I} a_i - \sum_{j=1 \atop a_j \in \sigma_I}^{m+e} l_j a_j = 0
\]
where $\sigma_I$ is the unique minimal cone containig $\sum_{i \in I} a_i$ (notice that $\underline{l}_I$ is not unique).

We claim that
\begin{equation}\label{eq:ineq1}
p_a(\underline{l}_I) \geq 0 \qquad \text{for} \quad a=1,\ldots,r\, .
\end{equation}

Recall that $p_a \in \mck$ and that $\mck$ is the image of $\Theta(CPL(\Sigma))$. Let $\varphi_a \in CPL(\Sigma)$ be a convex, piece-wise linear function such that $\Theta(\varphi_a)$ is a preimage of $p_a$. By the definition of $\Theta$ (cf. \eqref{def:Theta}) we have
\begin{align}
p_a(\underline{l}_I) &= \sum_{i \in I} \varphi_a(a_i) - \sum_{j=1 \atop a_j \in \sigma_I} l_j \varphi_a(a_j) = \sum_{i \in I} \varphi_a(a_i) - \varphi_a(\sum_{j=1 \atop a_j \in \sigma_I}l_j a_j) \notag \\
&\geq \varphi_a(\sum_{i \in I} a_i) - \varphi_a(\sum_{j=1 \atop a_j \in \sigma_I}l_j a_j) = 0\, . \notag 
\end{align}

Additionally, the following inequality 
\[
\# I = \sum_{i \in I} 1 \geq \sum_{l_j <0} l_j
\]
is true for the relation $\underline{l}_I$ , because again $\deg(a_i) =1$ for $i=1,\ldots,m$ and $\deg(a_j) \leq 1$ for $j=1,\ldots,m+e$ holds. Therefore the symbol of a box operator with respect to a generalized primitive collection is
\begin{align}
\sigma(\Box^X_{\underline{l}_I})  &= \prod_{a= 1}^r \chi_a^{p_a(\underline{l})} \prod_{j=1 \atop a_j \in \sigma_I}^{m+e}\sigma(\msd_j)^{l_j} - \prod_{i \in I} \sigma(\msd_i) \notag
\end{align}
if $\# I = \sum_{j=1 \atop a_j \in \sigma_I}^{m+e}$ and
\begin{align}
\sigma(\Box^X_{\underline{l}_I})  &= -\sum_{i \in I} \sigma(\msd_i) \notag 
\end{align}
if $\# I > \sum_{j=1 \atop a_j \in \sigma_I}^{m+e}$. For $\underline{\chi} = (\chi_1,\ldots, \chi_{r+e}) = 0$ this gives
\begin{equation}\label{eq:SigmaPrim}
\sigma(\Box^X_{\underline{l}_I})_{\mid z = \underline{\chi} = 0}  = -\sum_{i \in I} \sigma(\mathbf{D}_i)
\end{equation}
in both cases.

Notice that for $i=1,\ldots,r$ the $\mathbf{D}_i$, and therefore also the $\sigma(\mathbf{D}_i)$ satisfy
\begin{align}
\sum_{i=i}^m a_{ki} \mathbf{D}_i &= \sum_{i=1}^m a_{ki} \sum_{a=1}^{r} m_{ia} z \chi_a \p_{\chi_a} 
= \sum_{a=1}^r \left(\sum_{i=1}^m a_{ki} m_{ia} \right) z \chi_a \p_{\chi_a}  
= \;0\, . \label{eq:SigmaEuler}
\end{align}

If we keep Remark \ref{rem:classcoho} in mind, identify $\mfd_i$ with $\sigma(\mathbf{D_i})$ for $i=1,\ldots,r$ and use the relations \eqref{eq:SigmaCone}, \eqref{eq:SigmaEuler}, \eqref{eq:SigmaPrim} we see that 
\[
\dim_\mbc(S_{\mid z = \underline{\chi} = 0}) \leq H^*(X(\Sigma),\mbc)\cdot \prod_{j=m+1}^{m+e} l_j
\]
where here we denote by $l_j$ the $j$-th component of $\underline{l}_{C_j}$.

In conclusion, this shows that the variety $S$ over $\underline{\chi} = 0$ is zero dimensional, since its coordinate ring is a finite-dimensional vector space. This finishes the proof.

\end{proof}

\begin{prop}\label{prop:ReseqCoh}
Let $\qmclogo^{log,X}$ be the $\mcr$-module defined above. We have the following isomorphism of finite-dimensional commutative algebras:
\[
{\qmclogo^{log,X}}_{\mid z= \underline{\chi} =0} \simeq H_{orb}^{*}(X,\mbc)\, .
\]
\end{prop}
\begin{proof}
Let $l \in \mbl$ be a cone relation. The corresponding box operator in the limit $z=\underline{\chi}=0$ is equal to
\begin{align}
(\Box^X_{\underline{l}})_{\mid z=\underline{\chi} = 0} := & \prod_{i=1 \atop l_i < 0}^{m+e} \mathbf{D}_i^{-l_i} 
-  \prod_{i=1 \atop l_i >0}^{m+e}\mathbf{D}_i^{l_i} \label{eq:ConeRelLim}
\end{align}
where we have used the fact that for a cone relation $\underline{l}$ we have $p_a(\underline{l})=0$ for $a=1, \ldots ,r$.\\

Now suppose that $I \subset\{1, \ldots, m+e\}$ is a generalized primitive collection and consider a primitive relation $\underline{l}_I$:
\[
\sum_{i \in I} a_i - \sum_{j=1 \atop a_j \in \sigma_I}^{m+e} l_j a_j = 0\, .
\]
Notice that we have $p_a(\underline{l}_I) \geq 0$ for $a=1,\ldots,r$ which can be shown similarly to \eqref{eq:ineq1}.

We now claim that there exists an $a \in 1, \ldots ,r$ such that $p_a(\underline{l}_I) > 0$. Notice that the kernel of the map
\begin{align}
\mbl &\lra \mbz^r\, , \notag \\
\underline{l} &\mapsto (p_1(\underline{l}),\ldots, p_r(\underline{l}))\, . \notag
\end{align}
is $e$-dimensional, since $p_1, \ldots, p_r$ is part of a basis of $Pic^e(X)$. Because the $p_a$ vanish on all cone relations for $a=1,\ldots ,r$ (cf. \eqref{eq:vanConerel} and the space of cone relations has rank $e$, the claim follows by dimensional reasons. We therefore get
\begin{equation}\label{eq:PrimRelLim}
(\Box^X_{\underline{l}_I})_{\mid z=\underline{\chi} = 0} =  - \prod_{i\in I}\, . \mathbf{D}_i
\end{equation}

Using Lemma \ref{lem:orbcoho} and the formulas \eqref{eq:SigmaEuler}, \eqref{eq:ConeRelLim} and \eqref{eq:PrimRelLim}  we get the following surjective map
\begin{equation}\label{eq:mapCohQuot}
H^*_{orb}(X,\mbc) \twoheadrightarrow  {\qmclogo^{log,X}}_{\mid z= \underline{\chi} =0} \, ,
\end{equation}
by sending  $\mfd_i$ to $\mathbf{D}_i$.\\

Notice that ${\qmclogo^{log,X}}$ is coherent by the theorem above and its generic rank is equal to $vol(Q)$ by Proposition \ref{prop:PropLattice}. Since $H^*_{orb}(X,\mbc)$ is also $vol(Q)$-dimensional and the dimension of the fibers of a coherent sheaf is upper-semi-continuous, we conclude that the map above is an isomorphism.

\end{proof}

\begin{cor}\label{cor:locfree}
The $\mco_{\mbc_z \times \mcu_\mcx}$-module $\qmclogo^{log,X}$ is locally free of rank $\mu$.
\end{cor}
\begin{proof}
Since $D_\mcx \subset \mcm_\mcx$ is a normal crossing divisor it carries a natural stratification $\{S_i\}_{i \in I}$ by smooth subvarieties. The restriction of  $\qmclogo^{log,X}$  to $S_i$ is equipped with a $\mcd_{\mbc_z^* \times S_i}$-module structure, so that it must be locally free. Since each stratum contains $p_X$ in its closure the claim follows again by semi-continuity of the dimension of the fibers of a coherent sheaf and from Proposition \ref{prop:ReseqCoh} above.
\end{proof}

Let $\mce_\mcx$ be the restriction $(\qmclogo^{log,\mcx})_{\mid \mbc_z \times \{\underline{\chi}=0\}}$ and let $E_\mcx = \Gamma(\mbc_z, \mce_\mcx)$ be its module of global sections.

\begin{lem}\label{lem:extriv}
There is a canonical isomorphism
\begin{align}
\alpha_\mcx: \mco_{\mbc_z} \otimes_\mbc H^*_{orb}(\mcx,\mbc) &\overset{\simeq}{\lra} \mce_\mcx\, . \notag
\end{align}
It comes equipped with a connection
\begin{align}
&\nabla^{res,\underline{\chi}}: \mce_\mcx \lra \mce_\mcx \otimes z^{-2} \Omega^1_{\mbc_z} \notag
\end{align}
induced by the residue connection of $\nabla$. Let $\pi_\mcx: {_0}{\!\!^\circ}QM^{log,\mcx}_A \ra E_\mcx$  be the canonical projection. Let $F_\mcx := \pi_\mcx(\mbc[z \chi_1\p_{\chi_1}, \ldots, z \chi_{r} \p_{\chi_{r}}, z \p_{\chi_{r+1}}, \ldots , z \p_{\chi_{r+e}}]$ and denote by $\mcf_\mcx\subset \mce_\mcx$ the corresponding sheaf a $\mbc$-vector spaces. Then $\alpha_X(1 \otimes H^*_{orb}(\mcx,\mbc)) = \mcf_\mcx$. The connection operator $\nabla^{res,\underline{\chi}}_{\p_z}$ sends $\mcf_\mcx$ into $z^{-2}\mcf_\mcx \oplus  z^{-1} \mcf_\mcx$. 
\end{lem}
\begin{proof}
Recall that $E_\mcx$ is a quotient of $\mbc[z, z \chi_1 \p_{\chi_1}, \ldots z \chi_{r} \p_{\chi_{r}}, z \p_{\chi_{r+1}}, \ldots , z\p_{\chi_{r+e}}]$ and $E_\mcx /z E_\mcx$ is canonically isomorphic to $H^*_{orb}(\mcx,\mbc)$. Denote by $w_1, \ldots , w_\mu$ a basis of $H^*_{orb}(\mcx,\mbc)$ which can be represented as monomials in  $\mbc[z \chi_1 \p_{\chi_1}, \ldots z \chi_{r} \p_{\chi_{r}}, z \p_{\chi_{r+1}}, \ldots , z\p_{\chi_{r+e}}]$ of degree $d_1, \ldots , d_\mu$. Denote by $(E_\mcx)_{(0)}$ the localization of $E_\mcx$ at $0$. By Nakayama's lemma the basis  $w_1, \ldots , w_\mu$ lifts to a basis in $(E_\mcx)_{(0)}$ and hence provides a basis in a Zariski open neighborhood of $0\in \mbc_z$.  Since the $(w_i)$ are global sections we have to show that they are nowhere vanishing.

From the presentation of $\qmclogo^{log,\mcx}$ we see that
\begin{align}
&(z^2 \nabla^{res,\underline{\chi}}_{\p_z})\prod_{a=1}^{r}(z \chi_a \p_{\chi_a})^{k_a} \cdot \prod_{b=r+1}^{r+e}(z \p_{\chi_b})^{k_b} \notag \\
 =&(z^2 \p_z)\prod_{a=1}^{r}(z \chi_a \p_{\chi_a})^{k_a} \cdot \prod_{b=r+1}^{r+e}(z \p_{\chi_b})^{k_b} \notag \\
 =& \prod_{a=1}^{r}(z \chi_a \p_{\chi_a})^{k_a} \cdot \prod_{b=r+1}^{r+e}(z \p_{\chi_b})^{k_b} \cdot \left( (z^2 \p_z) + \sum_{c=1}^{r+e} k_c \cdot z \right) \notag \\
 =&\! \left(\! -\sum_{a=1}^{r+e} \sum_{i=1}^{m}m_{ia} z \chi_a \p_{\chi_a}\! +\! z \cdot\! \left(\sum_{a=1}^{r} k_a + \sum_{b=r+1}^{r+e} k_b \cdot \left(1- \sum_{i=1}^{m+e}   m_{ib} \right) \right) \right)\prod_{a=1}^{r}(z \chi_a \p_{\chi_a})^{k_a} \cdot \prod_{b=r+1}^{r+e}(z \p_{\chi_b})^{k_b}\, .  \notag
\end{align}
Hence, we have
\begin{equation}\label{eq: connlvl}
(z^2 \nabla_{\p_z}^{res,\underline{\chi}})(\underline{w}) = \underline{w} \cdot (A_0 + z A_{\infty})
\end{equation}
where $A_0,A_\infty \in M(\mu \times \mu,\mbc)$ and $A_\infty$ is a diagonal matrix with $d_1, \ldots , d_\mu$. Since the connection has no singularities in $\mbc_z^*$ we conclude that $\underline{w}$ is nowhere vanishing, hence is a $\mbc[z]$-basis of $E_X$. This contruction gives the isomorphism $\alpha_X$ which is of course independent of the choice of the basis $\underline{w}$.
\end{proof}

\begin{lem}\label{lem:DeligneExt}
Write $\qma^{log,\mcx}$for the restriction $({_0}\qma^{log,\mcx})_{\mid \mbc_z^* \times \mcu_\mcx}$. Then for any $a \in \{1,\ldots ,r\}$ the residue endomorphisms  
\[
z\chi_a \p_{\chi_a} \in  \mce nd_{\mco_{\mbc_z^*}}((\qma^{log,\mcx})_{\mid \mbc_z^* \times \{\underline{\chi}=0\}})
\]
are nilpotent.
\end{lem}
\begin{proof}
Under the identification of $(\qmclogo^{log,\mcx})_{z=\underline{\chi} = 0}$ with $H^*_{orb}(X,\mbc)$ the action of the operator $z \chi_a \p_{\chi_a}$ corresponds to the cup product with $\mfd_a$. Hence the class of $z \chi_a \p_{\chi_a}$ is nilpotent in $End_\mbc((\qmclogo^{log,\mcx})_{\mid z=\underline{\chi}=0})$. On the other hand,  the class of $z \chi_a \p_{\chi_a}$ gives rise to a well-defined element of $\mce nd_{\mco_z}\left( (\qmclogo^{log,\mcx})_{\underline{\chi}=0}\right)$, which is flat on $\mbc_z^*$ with respect to the residue connection. Its eigenvalues are algebraic functions on $\mbc_z$ which are constant on $\mbc_z^*$ and take the value zero at the origin. This implies that the eigenvalues are zero over all of $\mbc_z$, hence the residue endomorphisms are nilpotent as required.
\end{proof}

Denote by $D_\mcx$ the reduced normal-crossing divisor in $\mcu_\mcx$ given by $\{\chi_1 \cdot \ldots \cdot  \chi_r = 0\}$ and denote its components by $D_a$ for $a =1,\ldots,r$.
\begin{prop}\label{prop:logpair}
There is a non-degenerate flat $(-1)^n$-symmetric pairing
\[
P: \qmclogo^{log,\mcx} \otimes \iota^* \qmclogo^{log,\mcx} \lra z^n \mco_{\mbc_z \times \mcu_\mcx}
\]
i.e. $P$ is flat on $\mbc_z^* \times (\mcu_\mcx \setminus D_\mcx)$ and the induced pairings 
\[
P: (\qmclogo^{log})_{\mid z=0}  \otimes \iota^* (\qmclogo^{log})_{\mid z=0} \lra z^n \mco_{\mcu}
\]
and $P: (\qmclogo^{log})_{\mid D_a}  \otimes \iota^* (\qmclogo^{log})_{D_a} \lra z^n \mco_{\mbc_z \times D_a}$ are non-degenerate.
\end{prop}
\begin{proof}
Denote by $M_a$, $a=1,\ldots,r$, the unipotent monodromy  automorphism corresponding to a counter-clockwise loop around the divisor $\mbc^*\times D_a$ and by $M_z$ the monodromy automorphisms corresponding to a counter-clockwise loop arround $z=0$. Denote by $M_{z,u}$ resp. $M_{z,s}$ their unipotent and semi-simple part. We set $N_a=\log M_a$ and $N_z = \log M_{z,u}$. Denote by $H^\infty$ the space of multi-valued flat sections on which the monodromy operators $M_a$ and $M_z$ act. 
Let $f_1,\ldots,f_\mu$ be a basis of flat multi-valued sections of ${\qma}_{\mid \mbc_z^* \times \mcu_\mcx}$ which is adapted to the generalized eigenspace decomposition of the space $H^\infty$ with respect to the automorphisms $M_z$ and $M_a$. We define the single-valued sections 
\[
s_i =  e^{-\log z (\rho_i +\frac{N_z}{2 \pi i})} \prod_{a=1}^r e^{-\log q_a \frac{N_a}{2\pi i}} f_i
\]
for some $\rho_i$ such that $e^{2\pi i \rho_i}$ is the generalized eigenvalue of $f_i$ with respect to $M_z$. These sections provide a basis for ${\qmclogo^{\log,\mcx}}_{\mid \mbc_z^* \times \mcu_\mcx}$. Notice that $P(s_i,s_j)$ is holomorphic on $\mbc_z^*
  \times (\mcu_\mcx \setminus D_\mcx)$. By the flatness of $P$ we get that $P(s_i,s_j) = z^{-\rho_i - \rho_j}P(f_i,f_j)$ which shows that $P(s_i,s_j)$ extends over $\mbc_z^* \times \mcu_\mcx$ and is non-degenerate. Together with Proposition \ref{prop:PropzExt} we get a non-degenerate pairing on the restriction ${\qmclogo^{log,\mcx}}_{\mid (\mbc_z \times \mcu_\mcx) \setminus \{0\} \times D }$. Since, $\{0\} \times D_\mcx$ has codimension two in $\mbc_z \times \mcu_\mcx$, $P$ extends to a non-degenerate pairing on $\qmclogo^{log,\mcx}$.
\end{proof}

\begin{lem}
The induced pairing $P: E_X \otimes \iota^*E_X \ra z^n \mco_{\mbc_z}$, restricts to a pairing $P: F_X \times F_X \ra z^n \mbc$. The pairing $z^{-n}P$ on $F_X$ coincides under the identification made in Lemma \ref{lem:extriv} with the orbifold Poincar\'{e} pairing on $H^*_{orb}(X, \mbc)$ up to a non-zero constant. The corresponding statement holds if we replace $X$ by $Z$.
\end{lem}
\begin{proof}
Let $\{d_0,\ldots, d_t\} = \{q \in \mbq \mid H^{2q}_{orb}(\mcx,\mbc) \neq 0\}$ where $d_i < d_j$ for $i < j$. Set $r_k = dim H^{2d_k}_{orb}(X,\mbc)$ and notice that $d_0 = 0$, $d_t = n$ and $r_0 = r_t = 1$. Choose a homogeneous basis
\[
w_{1,d_0},w_{1,d_1}, \ldots, w_{r_1,d_1},\ldots,w_{1,d_{t-1}},\ldots,w_{r_{t-1},d_{t-1}},w_{1,d_t}
\]
where $w_{i,d_k} \in H^{2d_k}_{orb}(X,\mbc)$. Denote by $s_{1,d_0},s_{1,d_1}, \ldots, s_{r_1,d_1},\ldots,s_{1,d_{t-1}},\ldots,s_{r_{t-1},d_{t-1}},s_{1,d_t}$ the corresponding sections of $\mce_\mcx$ under the isomorphism $\alpha_\mcx$ of Lemma \ref{lem:extriv}. By Lemma \ref{lem:DeligneExt} we can find sections $\tilde{s}_{1,d_0},\tilde{s}_{1,d_1}, \ldots, \tilde{s}_{r_1,d_1},\ldots,\tilde{s}_{1,d_{t-1}},\ldots,\tilde{s}_{r_{t-1},d_{t-1}},\tilde{s}_{1,d_t}$ which satisfy $(\tilde{s}_{i,d_k})_{\mid \mbc_z \times \{\underline{\chi} = 0 \}} = s_{i,d_k}$ and 
\begin{align}
\nabla_{z \chi_a \p_{\chi_a}} \left(\prod_{c=1}^r e ^{\log \chi_c \frac{N_a}{2 \pi i}} \tilde{s}_{i,d_k}\right) &= 0 \qquad \text{for} \quad a =1, \ldots ,r\, ,\notag \\
\nabla_{z \p_{\chi_b}} \left(\prod_{c=1}^r e ^{\log \chi_c \frac{N_a}{2 \pi i}} \tilde{s}_{i,d_k}\right) &= 0 \qquad \text{for} \quad b = r+1, \ldots ,r+e\, .\notag
\end{align}
From the definition of the sections $\tilde{s}_{i,d_k}$ and the flatness of $P$ then follows
\[
P(\tilde{s}_{i,d_k},\tilde{s}_{j,d_l})(z,\underline{\chi}) = P (s_{i,d_k},s_{j,d_l})(z)
\]
and therefore
\begin{align}
0 &= z \chi_a \p_{\chi_a} P(\tilde{s}_{i,d_k},\tilde{s}_{j,d_l}) = P(\nabla_{z \chi_a \p_{\chi_a}} \tilde{s}_{i,d_k},\tilde{s}_{j,d_l}) - P(\tilde{s}_{i,d_k}, \nabla_{z\chi_a \p_{\chi_a}} \tilde{s}_{j,d_l})\, , \notag 
\\
0 &= z \p_{\chi_b} P(\tilde{s}_{i,d_k},\tilde{s}_{j,d_l}) = P(\nabla_{z \p_{\chi_b}} \tilde{s}_{i,d_k},\tilde{s}_{j,d_l}) - P(\tilde{s}_{i,d_k}, \nabla_{z \p_{\chi_b}} \tilde{s}_{j,d_l})\, . \notag
\end{align}
By continuity this holds on $\mbc_z \times \{\underline{\chi} = 0\}$.
This shows the multiplication invariance of the corresponding pairing on $H^*_{orb}(\mcx,\mbc)$. It follows from equation \ref{eq: connlvl} that
\begin{align}
z \nabla_{\partial_z}^{\mathit{res},\underline{q}}(w_{i,d_k}) &= d_k \cdot w_{i,d_k} + \frac{1}{z} \sum_{m=1}^{r_{k+1}}\Theta_{m,i,k} w_{m,d_k+1} \quad \text{for} \quad k < n\, , \notag \\
z \nabla_{\partial_z}^{\mathit{res},\underline{q}}(w_{1,d_t}) &= n \cdot w_{1,d_t \notag}\, ,
\end{align}
where $\Theta_{m,i,k} := (\check{A}_0)_{u,v}$ with $u= m + \sum_{l=1}^{k}r_l$ and $v = i + \sum_{l=1}^{k-1}r_l$ and $\check{A}_0$ is the matrix with respect to the basis $w_{1,0}, \ldots, w_{1,n}$ of the endomorphism $-c_1(X) \cup$. Since the pairing is multiplication invariant it is enough to show that $P(w_{1,d_0},w_{1,d_t}) \in z^n \mbc$. We have
\begin{align}
z \p_z P(w_{1,d_0},w_{1,d_t}) &= P (0 \cdot w_{1,d_0} + \frac{1}{z} \sum_{m=1}^{r_{l}}\Theta_{m,i,1} w_{m,d_1+1} , w_{1,d_t}) + P(w_{1,d_0} , n \cdot w_{1,d_t}) \notag \\
&= 0 + n \cdot P(w_{1,d_0},w_{1,d_t})\, . \notag
\end{align}
This shows $(z \p_z -n) P(w_{1,d_0},w_{1,d_t}) = 0$ and therefore $P(w_{1,d_0},w_{1,d_t})$ takes values in $z^n \mbc$.

It remains to show that the pairing $z^{-n}P$ coincides, under the isomorphism $\alpha:1\otimes H^*_{orb}(X,\mbc)\rightarrow F$ and possibly up to a
non-zero constant, with the Poincar\'e pairing on the cohomology algebra. First notice that by construction, $z^{-n}P$, seen as defined
on $H^*_{orb}(X,\mbc)$ is again multiplication invariant. It suffices now
to show that $P(1,a)$ equals the value of the Poincar\'e pairing on $1$ and $a$. But as we have seen above, $P(1,a)$ can only be non-zero
if $a\in H^{2n}_{orb}(X,\mbc)$. Since $dim H^{2n}_{orb}(X,\mbc) =1$, the $P$ on $H^*_{orb}(X,\mbc)$ is entirely determined by the non-zero complex number $P(w_{1,d_0},w_{1,d_t})$.\\
\end{proof}

\begin{prop}\label{prop:ExtensionInftyAtZero}
Consider the $\mco_{\mbc_z}$-module $\mce_X$ with the connection $\nabla^{\mathit{res},\underline{\chi}}$ and the subsheaf $\mcf_X \subset \mce_X$ from
lemma \ref{lem:extriv}.
\begin{enumerate}
\item
Let $\widehat{\mce}_X:=\mco_{\mbp^1_z\times\{\underline{0}\}} \otimes_\mbc F_X $ 
be an extension of $\mce_X$ to
a trivial $\mbp^1$-bundle. Then the connection $\nabla^{\mathit{res},\underline{\chi}}$ has a logarithmic pole at $z=\infty$ with spectrum (i.e.,
set of residue eigenvalues) equal to the (algebraic) degrees of the cohomology classes of $H^*_{orb}(X,\mbc)$. 
\item
The pairing $P$ on $\mce_X$ extends to a non-degenerate
pairing $P:\widehat{\mce}_X\otimes_{\mco_{\mbp^1}} \iota^*\widehat{\mce}_X \rightarrow \mco_{\mbp^1}(-n,n)$,
where $\mco_{\mbp^1}(a,b)$ is the subsheaf of $\mco_{\mbp^1}(*\{0,\infty\})$ consisting of meromorphic
functions with a pole of order $a$ at $0$ and a pole of order $b$ at $\infty$.
\end{enumerate}
\end{prop}
\begin{proof}

The formula \ref{eq: connlvl} shows that the connection $\nabla^{res,\underline{\chi}}$ has a logarithmic pole at $z= \infty$ and has residue eigenvalues  which are equal to the degrees  of the cohomology classes of $H^*_{orb}(\mcx)$. This shows the first point. The second point follows from Proposition \ref{prop:logpair} and the definition of $\hat{\mce}_X$.

\end{proof}
 Set $j_{\frac{1}{z}}:\mbc^*_{z} \ra \mbp^1_z\setminus\{0\}$ and $E^{\infty} := \psi_{\frac{1}{z}} j_{\frac{1}{z},!}( \mce_X^{an})^{\nabla^{res,\chi}}_{\mid \mbc_z^*}$ (where $\psi_{\frac{1}{z}}$ is the nearby cycle functor at $z = \infty$). It is known (cf. e.g. \cite[Lemma 7.6, Lemma 8.14]{He3} that there is a correspondance between logarithmic extensions of flat bundles and filtrations on the corresponding local system of flat sections.  With respect to the connection $(z^2 \nabla_{\p_z}^{res,\underline{\chi}})(\underline{w}) = \underline{w} \cdot (A_0 + z A_{\infty})$, the isomorphism 
\begin{equation} 
F_X \overset{\simeq}\lra E^\infty
\end{equation} 
  is given by multiplication with $z^{-A_\infty} z^{-A_0}$.
 
\begin{lem}\label{lem:resEndo}
The filtration $F_\bullet$ on $F_X$ is given by 
\[
F_{p} = \sum_{|k| \geq -p} \mbc \left( (z\chi_1\p_{\chi_1})^{k_1}\cdot \ldots \cdot (z \chi_r \p_{\chi_r})^{k_r} \cdot (z \p_{\chi_{r+1}})^{l_1} \cdot \ldots (z \p_{\chi_{r+e}})^{l_e}\right)\, .
\]
The residue endomorphism $N_a$ of $\qma^{log,\mcx}$ along $\mbc^*_z \times D_a$ acts  on $E^\infty$ and satisfies $N_a F_\bullet \subset F_{\bullet-1}$.
\end{lem}
\begin{proof}
The first claim follows from the identification of $H^*_{orb}(\mcx)$ with $F_X$ and the computation of the residue connection $(z^2 \nabla_{\p_z}^{res,\underline{\chi}})(\underline{w})$ in Lemma \ref{lem:extriv}.  The second claim is immediate since the residue endomorphism is induce by left multiplication with $z\chi_a \p_{\chi_a}$.
\end{proof}

The next result gives an extension of ${_0}\qma^{log}$ to a family of trivial $\mbp^1$-bundles, possibly after restricting
to a smaller analytic open subset inside $U^{an}$. Set $r:= \inf\{\mid \!\!\underline{\chi}\!\! \mid\, : \underline{\chi} \not \in \overline{\mcm}_\mcx^\circ\}$ and let $B:= B_r(0) \subset \mcu^{an}_X$ the open ball with radius $r$.

\begin{prop}\label{prop:LogTrTLEP}
There is an analytic open subset $V_X\subset \mcu_X^{an}$ still containing the point $p_X$ and a trivial holomorphic bundle ${_0}\wqma^{log,X} \rightarrow \mbp^1_z\times V_X$
such that
\begin{enumerate}
\item $({_0}\wqma^{log,X})_{|\mbc_z \times V_X} \cong ((\qmclogo^{log,})^{an})_{|\mbc_z \times V_X}$,
\item $({_0}\wqma^{log,X})_{|\mbp^1_z\times \{0\}} \cong (\widehat{\mce}_X)^{an}$,
\item The connection $\nabla$ has a logarithmic pole along $\widehat{D}_X$ on ${_0}\wqma^{log,X}$,
where $\widehat{D}_X$ is the normal crossing divisor $\left(\{z=\infty\}\cup\bigcup_{a=1}^{r} \{\chi_a=0\}\right)\cap \mbp^1_z\times V_X$,
\item
The given pairings $P:\qmclogo^{log,X} \otimes \iota^*\qmclogo^{log,X} \rightarrow z^n\mco_{\mbc_z\times \mcu_X^\circ}$ and $P:\widehat{\mce}_X\otimes_{\mco_{\mbp^1_z}} \iota^*\widehat{\mce}_X \rightarrow \mco_{\mbp^1_z}(-n,n)$ extend to a non-degenerate pairing
\[
P: {_0}\wqma^{log,X} \otimes_{\mco_{\mbp^1_z\times V_X}} \iota^* {_0}\wqma^{log,X} \rightarrow \mco_{\mbp^1_z\times V_X}(-n,n)\, ,
\]
where the latter sheaf is
defined as in point 3. of proposition \ref{prop:ExtensionInftyAtZero},
\item
The residue connection along $\frac{1}{z} = \tau = 0$
$$
\nabla^{\mathit{res},z=\infty}:{_0}\wqma^{log,X}/\tau \cdot {_0}\wqma^{log,X} \longrightarrow {_0}\wqma^{log,X}/\tau \cdot {_0}\wqma^{log,X}\otimes\Omega^1_{\{\infty\}\times V_X}(\log(\{\infty\}\times D)).
$$
has trivial monodromy around $\{\infty\}\times D$ and the  element of $1 \in F\subset H^0(\mbp^1_z\times U^0, {_0}\widehat{\qma}^{log,X})$ is horizontal for $\nabla^{res,z = \infty}$.
\end{enumerate}
\end{prop}
\begin{proof}
Set $\tilde{D}:= \bigcup_{a=1}^{r} \{\chi_a=0\}\cap B$. A logarithmic extension of $(\qma^{log,\mcx})^{an}_{\mid \mbc^*_z \times( B \setminus \tilde{D})}$ over\\ $(\{z=\infty\} \times B) \cup (\mbp^1_z \setminus \{0\} \times \tilde{D})$  is  given by a $\mbz^{r+1}$-filtration on the local system $\mcl = (\qma^{log,\mcx})^{an,\nabla}_{\mid \mbc^*_z \times( B \setminus \tilde{D})}$, which is split iff the extension is locally free (cf. \cite[Lemma 8.14]{He3}.\\

We are looking for an extension $\widehat{\qma} \ra (\mbp^1_z \setminus\{0\}) \times B$ which should satisfy two constraints. First, $\widehat{\qma}$ should restrict to
$(\qmclogo^{log,X})_{\mid \mbc^*_z \times B}$ on $\mbc^*_z \times B$ and second
it should restrict to $(\hat{\mce}_X)_{\mid \mbp_z^1\setminus \{0\}}$ over $\mbp_z^1 \setminus \{0\} \times \{\underline{\chi} = 0\}$.

The $\mbz^r$-filtration $P_\bullet$ corresponding to the extension over $\mbc^*_z \times \tilde{D}$ is trivial since its the Deligne extension due to Lemma \ref{lem:DeligneExt}. Let $L^\infty$ be the space of multi-valued flat sections of $(\qma^{log,\mcx})^{an}_{\mid \mbc^*_z \times( B \setminus \tilde{D})}$ and let $E^\infty$ be the space of multi-valued flat sections of $\mce^{an}_X$  from above.
 We have an isomorphism  $L^\infty \ra E^\infty$ which is given by multiplication with $\prod_{a=1}^r \chi_a^{N_a}$, where $N_a$ is the logarithm of the (uni-potent part of the ) monodromy, and restriction to $\{\underline{\chi} = 0 \}$. This allows us to shift the filtration $F_\bullet$ on $E^\infty$ (resp. $F_X$) to a filtration $F'_\bullet$ on $L^\infty$, which we denote by the same letter. This gives a $\mbz^{r+1}$-filtration $(F_\bullet, P_{\bullet})$ which is split, since $P_\bullet$ is trivial. The corresponding extension $\widehat{\qma}$ has logarithmic poles along $(\{ z = \infty\} \times B) \cup (\mbp^1 \setminus \{0 \} \times \tilde{D})$ and restricts to $(\qmclogo^{log,X})_{\mid \mbc^*_z \times B}$  on $\mbc_z^* \times B$ resp. $(\hat{\mce}_X)_{\mid \mbp^1_z \setminus \{0\}}$ on $\mbp^1_z \setminus \{0\} \times \{\underline{\chi} = 0\}$. We therefore can glue $\widehat{\qma}$ and $(\qmclogo^{log,X})_{\mid \mbc_z \times B}$ to a holomorphic bundle on $\mbp^1_z \times B$, which is trivial on on $\mbp^1_z \times \{\underline{\chi} = 0\}$ since its restriction is isomorphic to $\mce_X$. Since triviality is an open condition there exists a subset $V_X \subset B$  such that the restriction of the glued bundle to $\mbp^1_z \times V_X$ is trivial. This shows the points 1. to 3. . For the fourth point notice that the flat pairing $P$ gives rise to a pairing on $L^\infty$ which in turn gives rise to a pairing on $E^\infty$.  The pole order property of this pairing on $\hat{\mce}_X$ at $z= \infty$ can be encoded by an orthogonality property of the filtration $F_\bullet$ with respect to that pairing (see .e.g. \cite[Theorem 7.17, Definition 7.18]{He4}). Hence the same property must hold  for $P$ and $F_\bullet$ seen as defined on $L^\infty$, so we conclude $P: {_0}\wqma^{log,X} \otimes_{\mco_{\mbp^1_z\times V_X}} \iota^* {_0}\wqma^{log,X} \rightarrow \mco_{\mbp^1_z\times V_X}(-n,n)$ as required.

The last statement follows from the fact that the residue connection $\nabla^{res, z = \infty}$  defined on ${_0}\wqma^{log,X} / z^{-1} {_0}\wqma^{log,X}$ has trivial monodromy  around $D_X  \cap \{z=\infty\} \times B$ if for any $a=1 ,\ldots ,r$ the nilpotent part $N_a$ of the monodromy  of $\mcl$  kills $gr^{F'}_\bullet$, i.e. $N_a F_\bullet \subset F_{\bullet -1}$. Using the identification $(L^\infty, F'_\bullet)$ with $(E^\infty,F^\bullet)$ this has been shown in Lemma \ref{lem:resEndo}. 

 From this follows that the element $1$ is a global sections over $\mbp^1_z \times V_X$ and flat with respect to the residue connection.
 
\end{proof}
\subsection{Frobenius structures}

We begin with a definition from \cite{Re1} which formalizes the structure which we obtained in Proposition \ref{prop:LogTrTLEP}.

\begin{defn} Let $M$ be a complex manifold of dimension bigger or equal than one and $D \subset M$ be a simple normal crossing divisor.
\begin{enumerate}
\item A $\TEPlog$-structure on $M$ is a holomorphic vector bundle $\mch \ra \mbp_z^1\times M$ which is equipped with an integrable connection $\nabla$ with a pole of type one along $\{0\} \times M$ and a logarithmic pole along $(\mbc_z \times D)$ and a flat, $(-1)^n$-symmetric, non-degenerate pairing
$P:\mch \otimes \iota^*\mch \rightarrow z^n \mco_{\mbc_z\times M}$.
If $D$ is empty we will simply denote it as a $\TEP$-structure.

\item A $\trTLEPlog$-structure on $M$ is a holomorphic vector bundle $\widehat{\mch} \ra \mbp_z^1\times M$ such that $p^*p_* \widehat{\mch} = \widehat{\mch}$ (where $p:\mbp_z^1\times M \twoheadrightarrow M$ is the projection) which is equipped with an integrable connection $\nabla$ with a pole of type $1$ along $\{0\}\times M$ and a logarithmic pole along $(\mbp^1_z \times D) \cup (\{\infty\}\times M)$ and a flat, $(-1)^n$-symmetric, non-degenerate pairing
$P:\widehat{\mch} \otimes \iota^*\widehat{\mch} \rightarrow \mco_{\mbp^1_z\times M}(-n,n)$.
If $D$ is empty we will simply denote it as a $\trTLEP$-structure.
\end{enumerate}
\end{defn}

\begin{prop}
Let $\mcx(\mathbf{\Sigma})$ be a projective toric Deligne-Mumford stack with an $S$-extended stacky fan $\mathbf{\Sigma}^e$ with $S = Gen(\Sigma)$ and let $W:Y \times \mcm_\mcx \lra \mbc \times \mcm_\mcx$ the corresponing Landau-Ginzburg model. Then the tuple $({_0}\wqma^{log,X}, \nabla, P)$ from Proposition \ref{prop:LogTrTLEP} is a $\trTLEPlog$-structure on $V_X \subset \mcm_\mcx^{an}$.
\end{prop}
\begin{proof}
This follows from Proposition \ref{prop:LogTrTLEP}.
\end{proof}

The following theorem which is a combination of Proposition 1.20 and Theorem 1.22 in \cite{Re1} gives sufficient conditions when a given $\trTLEPlog$-structure can be unfolded to a logarithmic Frobenius manifold.

\begin{thm}\label{thm:ReUnfold}
Let $(M,0)$ be a germ of a complex manifold and $(D,0) \subset (M,0)$ be a normal crossing divisor. Let $(\mch,0),\nabla.P)$ be a germ of a $\trTLEPlog$-structure on $\mbp^1 \times (M,0)$. Suppose that there is a section $\xi \in H^0(\mbp^1 \times (N,0), \mch)$ whose restriction to $\{\infty\} \times(N,0)$ is horizontal for the residue connection $\nabla^{res}: \mch /z^{-1} \mch  \ra \mch / z^{-1} \mch \otimes \Omega^1_{\{\infty\} \times N}(log(\infty\} \times D))$ and which satisfies the conditions
\begin{enumerate}
\item[(IC)] The map $\Theta(log D)_{\mid  0} \ra p_*\mch_{\mid 0 }$ induced by $[Z \nabla_\bullet](\xi): \Theta(log D) \ra p_* \mch$ is injective.
\item[(GC)] The vector space $p_* \mch_{\mid 0}$ is generated by $\xi$  and its images under iteration of the maps $\mcu$ and $[z\nabla_X]$ for any $X \in \Theta(log) D$.
\item[(EC)] $\xi$ is an eigenvector for the residue endomorphism $\mcv \in \mce nd_{\mco_{\{\infty\} \times M}}(\mch / z^{-1} \mch)$.
\end{enumerate}
Then there exists a unique (up to canonical isomorphism) gern of a logarithmic Frobenius manfold on $(\tilde{M},\tilde{D})$ with a unique embedding $i: M \hookrightarrow \tilde{M}$ with $i(M) \cap \tilde{D} = i(D)$ and a unique isomorphism $\mch \ra p^*\Theta_{\tilde{M}}(log \tilde{D})_{\mid i(M)}$ of $\trTLEPlog$-structures.
\end{thm}

Using the theorem above we are now able to construct a logarithmic Frobenius manifold from the Landau-Ginzburg model corresponding to a projective toric Deligne-Mumford stack.

\begin{thm}\label{thm:logfrobB}
Let $W: Y \times \mcm_\mcx \ra \mbc \times \mcm_\mcx$ be the Landau-Ginzburg model corresponding to a projective toric Deligne-Mumford stack. Then there exists a canonical logarithmic Frobenius manifold on $(V_X \times \mbc^{\mu -r},0)$ with logarithmic pole along $(D  \times \mbc^{\mu-r},0)$.
\end{thm}
\begin{proof}
In order to apply Theorem \ref{thm:ReUnfold} to the $\trTLEPlog$-structure obtained in Proposition \ref{prop:LogTrTLEP} we define the section $\xi$ to be the class of $1$. Because of Proposition \ref{prop:LogTrTLEP} 5. this section is flat with repsect to the residue connection along $\frac{1}{z} = \tau = 0$. The conditions $(IC)$ and $(GC)$ follow from the identification of $(\qmclogo^{log,X})_{\mid 0}$ with the cohomology ring $(H^{*}_{orb}(X_\Sigma,\mbc), \cup)$ (cf. Proposition \ref{prop:ReseqCoh} and Formula \ref{eq:mapCohQuot}), the definition of the $\mathbf{D}_i$ for $i=1,\ldots, n$ (cf. Formula \ref{eq:degDi}) and the representation of $H^*_{orb}(\mcx_\Sigma,\mbc)$ in Lemma \ref{lem:orbcoho}. The condition $(EC)$ follows from Proposition \ref{prop:ExtensionInftyAtZero} 1.
\end{proof}

\section{Orbifold Quantum cohomology}\label{sec:QC}

In this section we review some constructions from orbifold quantum cohomology.

Let $\mcx$ be a smooth proper Deligne-Mumford stack over $\mbc$. The inertia stack of $\mcx$ is defined
\[
I\mcx := \mcx \times_{\mcx \times \mcx} \mcx
\]
with respect to the diagonal morphism $\Delta: \mcx  \ra \mcx \times \mcx$. A geometric point  on $I \mcx$ is given by a geometric point $x \in \mcx$ and an element $g \in Aut(\mcx)$ of the isotropy group. We call $g$ the stabilizer of $(x,g) \in I \mcx$.  The inertia stack is a smooth Deligne-Mumford stack but different components will in general have different dimensions. Let $T$ be the the index set of the components of $I \mcx$. Let $0 \in T$ be the distinguished element corresponding to the trivial stabilizer. We thus have
\[
\mci = \bigsqcup_{v \in T} \mcx_v\, .
\]
The orbifold cohomology of $\mcx$ is defined, as a vector space, by $H^*_{orb}(\mcx,\mbc) := H^*(I \mcx, \mbc)$, hence we have
\[
H^*_{orb}(\mcx, \mbc) = H^*(\mcx,\mbc) \oplus \bigoplus_{v \in T'} H^*(\mcx_v, \mbc)
\]
where $T' := T \setminus \{ 0\}$ is the index set of the twisted sectors.

In order to define a grading on the orbifold cohomology , we associate to any $v \in T$ a rational number called the age of $\mcx_v$.

 The genus zero Gromov-Witten invariants with descendants are defined by
\[
\langle \alpha_1 \psi_1^{k_1},\ldots, \alpha_l \psi_l^{k_l}\rangle_{0,n,d} := \int_{[\overline{\mcm}_{g,n}(\mcx,d)]^{vir}} \prod_{i=1}^l ev^*_i(\alpha_i)\psi_i^{k_i}
\]
where $\alpha_i \in H^*_{orb}(\mcx)$, $d \in H_2(X,\mbz)$, $k_i$ is a non-negative integer, $\overline{\mcm}_{0,l}(\mcx,d)$ is the moduli stack of genus zero, $l$-pointed stable maps to $\mcx$ of degree $d$, $[\overline{\mcm}_{0,l}(\mcx,d)]^{vir}$ is the virtual fundamental class, $ev_i$ is the evaluation map at the $i$-th marked point
\[
ev_i: \overline{\mcm}_{0,l}(\mcx,d) \ra  I \mcx
\]
and $\psi_i = c_1(L_i)$ where $L_i$ is the line bundle over $\overline{\mcm}_{0,l}(\mcx,d)$ whose fiber at a stable map is the cotagent space of the coarse curve at the $i$-th marked point. The correlator $\langle \alpha_1 \psi_1^{k_1},\ldots, \alpha_l \psi_l^{k_l}\rangle_{0,n,d}$ is non-zero only if  $d \in \text{Eff}_\mcx \subset H_2(X;\mbz)$, where  $\text{Eff}_\mcx$ is the semigroup generated by effective stable maps.

We choose a homogeneous basis $T_0,\ldots, T_s$ of $H^{\ast}_{orb}(\mcx)$, where $T_0 =1 \in H^0(\mcx,\mbc)$, $T_1,\ldots, T_r \in H^{2}(\mcx)$ and $T_i \in  \bigoplus_{k \neq 0,2}H^{k}(\mcx) \oplus \bigoplus_{v \in T'} H^{*}(\mcx_v)$. We denote by $T^0,\ldots, T^s$  the basis of $H^*(\mcx)$ which is dual with respect to the orbifold Poincar\'{e} pairing.

Let $\alpha, \beta , \tau \in H^*_{orb}(\mcx,\mbc)$ and write $\tau  = \tau' + \delta$ where $\delta \in H^2(\mcx,\mbc)$ and $\tau' \in \bigoplus_{k \neq 2} H^{k}(\mcx) \oplus \bigoplus_{v \in T'} H^*(\mcx_v)$. We define the the big orbifold quantum product $\circ_\tau$ as the formal family of commutative and associative products on $H^*_{orb}(\mcx) \otimes \mbc\llbracket \text{Eff}_\mcx \rrbracket$:
\begin{align}
\alpha \circ \gamma &:= \sum_{d \in \text{Eff}_\mcx} \sum_{l,k \geq 0} \frac{1}{l!}\langle \alpha,\gamma, \underbrace{\tau,\ldots,\tau}_{l- \text{times}},T_k\rangle_{0,l+3,d}\,T^k Q^d \notag \\
&= \sum_{d \in \text{Eff}_\mcx} \sum_{l,k \geq 0} \frac{e^{\delta(d)}}{l!}\langle \alpha,\gamma, \underbrace{\tau',\ldots,\tau'}_{l- \text{times}},T_k\rangle_{0,l+3,d}\,T^k Q^d \notag
\end{align}
where the last equality follows from the divisor axiom. The Novikov ring $\text{Eff}_\mcx$ was introduced to split the contribution of the different $d\in \textup{Eff}_\mcx$. However, we will make the following assumption:
\begin{assumption}
The orbifold quantum product $\circ_\tau$ is convergent over an open subset $U \subset H^*_{orb}(\mcx)$:
\[
U = \{\tau \in H^*_{orb}(\mcx) \mid \Re(\delta(d)) < -M, \forall d \in \text{Eff}_\mcx \setminus \{0\}, ||\tau'|| <e^{-M}\}
\]
for some $M \gg 0$ (here $||\cdot||$ is the standard hermitan norm on $H^*_{orb}(\mcx)$).
\end{assumption}
Using this assumption, we can set $Q=1$. We will denote this product on $H^*_{orb}(\mcx,\mbc)$ parametrized by $\tau \in U$ by $(H^*(\mcx,\mbc),\circ_\tau)$.

Let $t_0,\ldots, t_s$ be the coordinates on $H^*_{orb}(\mcx)$ determined by the homogeneous basis.
\begin{defn}\label{defn:givconn}
The Givental connection is the tuple $(\mcf^{big},\nabla^{Giv}, P)$ which consists of the trivial holomorphic vector bundle $\mcf^{big} := H^*(\mcx,\mbc) \times (U \times \mbp^1_z)$, the connection $\nabla^{Giv}$
\begin{align}
\nabla_{\p_{t_k}} &:= \frac{\p}{\p_{t_k}} - \frac{1}{z} T_k \circ_\tau\, , \notag \\
\nabla_{z \p_z} &:= z\frac{\p}{\p_z} + \frac{1}{z}E \circ_{\tau} + \mu \notag
\end{align}
where $\mu: H^{*}_{orb}(\mcx,\mbc) \ra H^{*}_{orb}(\mcx,\mbc)$ is the grading operator given by $\mu(T_k) =deg(T_k)/2$ and the holomorphic Euler vector field $E$ is given by
\[
c_1(T\mcx) + \sum_{k=1}^N\left(1-\frac{deg(T_i)}{2} \right)t_k T_k
\]
and the pairing
\begin{align}
P: \mcf^{big} \otimes \iota^* \mcf^{big} &\lra \mco_{\mbp^1_z \times U}(-n,n)\, , \notag \\
(a,b) &\mapsto z^n(a(z),b(-z))_{orb}\notag
\end{align}
where $\iota(z,\underline{t}) = (-z,\underline{t})$ and $(-,-)_{orb}$ is the Orbifold Poincar\'{e} pairing.
\end{defn}
Notice that the connection $\nabla^{Giv}$ is flat (\cf. \cite[$\S$ 2.2]{Ir2}) and the pairing $P$ is non-degenerate, $(-1)^n$-symmetric and $\nabla^{Giv}$-flat.\\

Let $H^{gen}_{orb}(\mcx) \supset H^2(\mcx)$ be a minimal homogeneous subspace which generates $H^*_{orb}(\mcx)$ with respect to the orbifold cup-product. We write $H^{gen}_{orb}(\mcx) = H^2(\mcx) \oplus H'_{orb}(\mcx)$.

\begin{defn}
\begin{enumerate}
\item Let $\alpha, \gamma \in H^*_{orb}(\mcx)$ and $\tau \in H^{gen}_{orb}(\mcx) \cap U$. Define the semi-small quantum product as the restricton of the quantum product to parameter space $H^{gen}_{orb}(\mcx)$:
\[
\alpha \circ \gamma 
= \sum_{d \in \text{Eff}_\mcx} \sum_{l,k \geq 0} \frac{e^{\delta(d)}}{l!}\langle \alpha,\gamma, \underbrace{\tau',\ldots,\tau'}_{l- \text{times}},T_k\rangle_{0,l+3,d}\,T^k Q^d  
\]
for $\tau = \delta + \tau' \in H^2(\mcx)\, \oplus\, H'_{orb}(\mcx)$.
\item The semi-small Givental connction $(\mcf^{ss},\nabla^{Giv},P)$ is the  restriction of the Givental connection to $(H^{gen}_{\orb}(\mcx) \cap U) \times\mbp^1_z$.
\end{enumerate}
\end{defn}

Let $L_\xi \ra \mcx$ be a orbifold line bundle corresponding to $\xi \in Pic(\mcx)$. For any point $(x,g) \in \mcx_v \subset \mci \mcx$ the stabilizer $g$ acts on the fiber $L_x$ by a rational number. This number depends only on the sector $v$, hence we denote the number by $f_v(\xi)$ and call it the age of $L_\xi$ along $\mcx_v$.

Iritani defined an action of $Pic(\mcx)$ on $(\mcf^{big}, \nabla^{Giv},P)$ and showed that it is equivariant with respect to this action:
\begin{prop}
For each $\xi \in Pic(\mcx)$ there is an isomorphism of $\mcg^{big}$
\begin{align}
H^*_{orb}(\mcx,\mbc) \times (U \times \mbc) &\lra H^*_{orb}(\mcx,\mbc) \times (U \times \mbc)\, , \notag \\
(\alpha,\tau,z) &\mapsto (dG(\xi)\alpha,G(\xi)\tau,z) \notag
\end{align}
which preserves the connection $\nabla^{Giv}$ and the pairing $P$, where $G(\xi),dG(\xi):H^*_{orb}(\mcx) \lra H^*_{orb}(\mcx)$ are defined by
\begin{align}
G(\xi)(\tau_0 + \sum_{v \in T'} \tau_v) &= (\tau_0 - 2\pi i \xi_0) + \sum_{v \in T'}e^{2 \pi i f_v(\xi)} \tau_v\, , \notag \\
dG(\xi)(\tau_0 + \sum_{v \in T'} \tau_v) &= \tau_0 + \sum_{v \in T'} e^{2 \pi i f_v(\xi)} \tau_v \notag
\end{align}
where $\tau_v \in H^*(\mcx_v)$ and $\xi_0$ is the image of $\xi$ in $H^2(\mcx,\mbq)$.
\end{prop}

It follows from the Proposition above that the Givental connection is invariant under the action of $Pic(\mcx)$, however,  as observed in \cite{Douai-Mann}, the functions $t_0,\chi_1=e^{t_1},\ldots \chi_r=e^{t_r},t_{r+1},\ldots , t_s$ are not coordinates on $H^*_{orb}(\mcx,\mbc)/ Pic(\mcx)$. Therefore we mod out only a subgroup namely the subgroup $Pic(X)$ of line bundles with zero age , i.e. 
$f_v(\xi)= 0$. The set $U$ is invariant under the action of $Pic(X)$.\\

Let $V$ be the quotient of $U$ by the action of $Pic(X)$ and denote by $\pi : U \ra V$ the natural projection. Set $\chi_i = e^{t_i}$ for $i=1,\ldots ,r$, then $t_0,\chi_1,\ldots,\chi_{r},t_{r+1},\ldots , t_s$ are coordinates for $V$.
\begin{lem}$ $\\[-1em]
\begin{enumerate}
\item There is a trTLEP(n)-structure $(\mcg^{big}, \nabla^{Giv},P)$ on $V$ s.t. $\pi^*(\mcg^{big},\nabla^{big},P) = (\mcf^{big}, \nabla^{big},P)$.
\item Set $V_{gen} := \pi(H^{gen}_{\orb}(\mcx) \cap U)$.  There is a trTLEP(n)-structure $(\mcg^{ss},\nabla^{Giv},P)$ on $V_{gen}$ s.t. $\pi^*(\mcg^{ss},\nabla^{Giv},P) = (\mcf^{ss},\nabla^{Giv},P)$ and $(\mcg^{ss},\nabla^{Giv},P) = (\mcg^{big},\nabla^{Giv},P)_{\mid V_{gen}}$.
\end{enumerate}
\end{lem}
\begin{proof}
The statements are a direct consequence of Proposition 4.4. The connection of $(\mcg^{big},\nabla^{big},P)$ is given by
\begin{align}
\nabla_{\p_{t_k}} &:= \frac{\p}{\p_{t_k}} - \frac{1}{z} T_k \circ_\kappa\, , \notag \\
\nabla_{\chi_j \p_{\chi_j}} &:= \chi_j\frac{\p}{\p_{\chi_j}} - \frac{1}{z} T_j \circ_\kappa \qquad \text{for} \quad j = 1,\ldots ,r\, ,\label{eq:logconn} \\
\nabla_{z \p_z} &:= z\frac{\p}{\p_z} + \frac{1}{z}E \circ_{\kappa} + \mu \notag
\end{align}
for $\kappa \in V$.
\end{proof}
Let $T_0,\ldots, T_s$ be a homogeneous basis of $H^*_{orb}(\mcx)$ as above. We assume that $1=T_0\in H^0(\mcx,\mbz)$,  $T_1,\ldots, T_r \in H^2(\mcx)$ and $T_{r+1}, \ldots , T_{s}$ is a basis of $\bigoplus_{k \neq 0,2} H^{k}(\mcx)\oplus \bigoplus_{v \in T'} H^*(\mcx_{(v)})$.
Additionally we assume that $T_1,\ldots , T_{r}$ is a $\mbz$-basis of $Pic(X)\subset H^2(\mcx,\mbz)$ and lies in the K\"ahler cone $\overline{\mck} \subset H^2(\mcx)$.\\

The choice of the basis $T_0,\ldots, T_s$ gives rise to an embedding $j:H^2(X,\mbc)/Pic(X) \hookrightarrow \mbc^r$. Let $V'_{gen}$ resp. $V'$ be the closure of image of $j\times id$.
\begin{prop}\label{prop:logFrobA}
There exist extensions $(\overline{\mcg}^{big},\hat{\nabla}^{Giv},P)$ resp. $(\overline{\mcg}^{ss}, \hat{\nabla}^{Giv},P)$ of $(\mcg^{big},\hat{\nabla}^{Giv},P)$ resp. $(\mcg^{ss}, \hat{\nabla}^{Giv},P)$ to a log-trTLEP(n)-structure on $V'$ resp. $V'_{gen}$. Moreover, there is a structure of a logarithmic Frobenius manifold on $V'$.
\end{prop}
\begin{proof}
The first statement follows from the form of the connection \ref{eq:logconn}.
The second statement follows from \cite[Proposition 1.10 and Proposition 1.11]{Re1}, where the vector $\xi$ in loc. cit. corresponds to $T_0 =1$ here.
\end{proof}

We now recall the fundamental solution of the Givental connection. Define
\[
L(\tau,z)\alpha := e^{-\delta/z} \alpha - \sum_{d \in Eff_\mcx \setminus \{0\} \atop l > 0, 0\leq k \leq s} \frac{1}{l!}\langle \frac{e^{-\delta/z}\alpha}{z+\psi},\tau',\ldots, \tau', T_k\rangle_{0,l+2,d}e^{\delta(d)}T^k
\]
where $\tau = \delta + \tau'$.
The following proposition summarizes the properties of the fundamental solution.
\begin{prop}[\protect{\cite[Proposition 2.4]{Ir2}}]
$ $\\[-1em]
\begin{enumerate}
\item $L(\tau,z)$ satisfies the following differential equations:
\[
\nabla_{\p_{t_k}} L(\tau,z) \alpha = 0, \quad \nabla_{z \p_z} L(\tau,z)\alpha = L(\tau,z)(\mu \alpha - \frac{\rho}{z}\alpha)
\]
where $\alpha \in H^*_{\orb}(\mcx)$, $\rho := c_1(TX) \in H^2(\mcx)$ and $\mu$ is the grading operator from Definition \ref{defn:givconn}. If we put $z^{-\mu}z^\rho := \exp(-\mu \log z) \exp(\rho \log z)$, then
\[
\nabla_{\p_{t_k}} L(\tau,z) z^{-\mu}z^\rho \alpha = 0, \quad \nabla_{z \p_z} L(\tau,z)z^{-\mu}z^\rho\alpha =0\, .
\]
\item $L(\tau,z)$ is convergent and invertible on $U \times \mbc^*$.
\item $(L(\tau,-z)\alpha, L(\tau,z)\beta)_{orb} = (\alpha,\beta)_{orb}$
\item $dG(\xi)L(G(\xi)^{-1}\tau,z)\alpha = L(\tau,z) e^{2\pi i \xi_0} e^{2\pi i f_v(\xi)} \alpha$ for $\alpha \in H^*(\mcx_v)$. In particular
\[
dG(\xi)L(G(\xi)^{-1}\tau,z)\alpha = L(\tau,z)\alpha
\]
for $\xi \in Pic(X)$.
\item Define $\tilde{L}(\tau,z) := L(\tau,z) z^{-\mu}z^\rho$, then
\[
\nabla_{\p_{t_k}} \tilde{L}(\tau,z) \alpha = 0, \quad \nabla_{z \p_z} \tilde{L}(\tau,z)\alpha = 0\, .
\]
\end{enumerate}
\end{prop}
Since $L$ and therefore also $\tilde{L}$ is invertible, then the sections $s_i:= \tilde{L}(T_i)$ are a basis of flat sections.

The $J$-function of $\mcx$ is given by $L(\tau,z)^{-1} 1 = L(\tau,z)^{-1} T_0$. We set 
\[
\tilde{J}:= \sum_{i=0}^s \tilde{J}_i T_i := \sum_{i=0}^s (s_i,T_0)_{orb} T_i = \tilde{L}(\tau,z)^{-1}(T_0)
\]
and get the formula
\begin{equation}\label{eq:Jfuncterms}
1= T_0 = \sum_{i=0}^s \tilde{J}_i s_i \qquad \text{on} \quad \mbc_z^* \times U\, .
\end{equation}

\section{Mirror correspondance}

Let $\mcx$ be projective, toric orbifold. In order to state the mirror theorem for toric orbifolds we have to introduce the $I$-function.

\begin{defn}
The $I$-function of a toric orbifold $\mcx$ is defined by
\[
I(\underline{\chi},z) = e^{\sum_{a=1}^{r+e} e^{\overline{p}_a} \log \chi_a/z}\sum_{d \in \mbk} \chi^d \frac{\prod_{\nu= \lceil \langle D_i,d \rceil}^\infty(\overline{D}_i +(\langle D_i,d \rangle -\nu)z)}{\prod_{\nu =0}^\infty (\overline{D}_i +(\langle D_i,d \rangle -\nu)z)}\mathbf{1}_{v(d)}\, .
\]
\end{defn}
We collect a few facts about the $I$-function.
\begin{lem}\label{lem:propI}$ $\\[-1em]
\begin{enumerate}
\item $e^{-\sum_{a=1}^{r+e} e^{\overline{p}_a} \log \chi_a/z}I(\underline{\chi},z) \in H^*_{orb}(\mcx)[z,z^{-1}]\llbracket \chi_1,\ldots,\chi_r\rrbracket$.
\item The function $e^{-\sum_{a=1}^{r+e} e^{\overline{p}_a} \log \chi_a/z}I(\underline{\chi},z)$ is a convergent power series in $\chi_1,\ldots, \chi_r$ if and only if $\rho \in \mck^e$. In this case, the $I$-function has the asymptotics
\[
I(\underline{\chi},z) = 1 + \frac{\tau(\underline{\chi})}{z} + o(z^{-1})\, .
\]
The function $\tau$, which take values in $H^{\leq 2}_{orb}(\mcx)$, is a local embedding and is called the mirror map.
\item 
Set $\tilde{I} := I z^{-\rho} z^{\mu}$ then
\[
\check{E} (\tilde{I}) = 0 \quad \text{and} \quad   \Box_{\underline{l}}^X (\tilde{I}) = 0 \quad \text{for} \quad \underline{l} \in \mbl\, . \notag
\]
\end{enumerate}
\end{lem}
\begin{proof}
The first point follows directly from the definition of the $I$-function. The second point is \cite[Lemma 4.2]{Ir2}. The third point follows from $\tilde{\Box}_{\underline{l}} = \prod_{i=1}^e \chi_{r+i}^{|l_{m+i}|}\cdot\Box^X_{\underline{l}}$ (cf. Lemma \ref{lem:facBox}) and the \cite[lemma 4.19]{Ir2} (Note that Iritani proves this for the equivariant $I$-function. In order to to get the statement one simply has to consider the equivariant limit $\lambda \ra 0$).
\end{proof}

There is the following theorem which compares the $I$-function from above  and the $J$-function which has been introduced in Section \ref{sec:QC}.
\begin{thm}\label{thm:mirthm}
Let $\rho \in \mck^e$, then the $I$-function and the $J$-function coincide up to a coordinate change given by the mirror map $\tau$,i.e.
\[
I(\underline{\chi},z) = J(\tau(\underline{\chi}),z)\, .
\] 
\end{thm}

We can now identify the two  $\trTLEPlog$-structures  $({_0}\wqma^{log,X}, \nabla, P)$ and $(\mcg^{ss},\nabla^{Giv},P)$.

\begin{prop}\label{prop:isolog}
There exists a small analytic neighborhood $W_X$ of $0$ in $V_X$ such that there is an isomorphism
\[
\theta: ({_0}\wqma^{log,X})_{\mid \mbp^1_z \times W_X} \lra (id_{\mbp^1_z} \times \tau)^* \mcg^{ss}_{\mid \mbp^1_z \times W_X}
\]
of $\trTLEPlog$-structures on $W_X$.
\end{prop}
\begin{proof}
As a first step  we define a morphism of holomorphic vector bundles with meromorphic connections
\begin{align}
\gamma: \left((\qmclogo^{log,X})^{an}_{\mid\mbc_z \times W_X}, \nabla\right) &\lra (id_{\mbc_z} \times \tau)^* \left(\mcg^{ss}_{\mid \mbc_z \times W_X}, \hat{\nabla}^{Giv} \right)\, , \notag \\
1 &\mapsto 1 = T_0\, . \notag
\end{align}
 We set $\tilde{\Box}^X_{\underline{l}} := (id_{\mbc_z }\times \tau)_* \Box^X_{\underline{l}}$ and $\tilde{E} := (id_{\mbc_z} \times \tau)_* \check{E}$. In order to show that the morphism above is well-defined, the following equations have to hold:
\begin{align}
\tilde{\Box}^X_{\underline{l}}(\chi_{1},\ldots,\chi_{r+e},z, \hat{\nabla}^{Giv}_{z \chi_1 \p_{\chi_1}}, \ldots, \hat{\nabla}^{Giv}_{z \chi_r \p_{\chi_r}}, \hat{\nabla}^{Giv}_{z \p_{\chi_{r+1}}}, \ldots, \hat{\nabla}^{Giv}_{z \p_{\chi_{r+e}}}) (1) &= 0 \quad \text{for all} \; \underline{l} \in \mbl\, , \notag \\
\tilde{E}(\chi_{1},\ldots,\chi_{r+e},z, \hat{\nabla}^{Giv}_{z \chi_1 \p_{\chi_1}}, \ldots, \hat{\nabla}^{Giv}_{z \chi_r \p_{\chi_r}}, \hat{\nabla}^{Giv}_{z \p_{\chi_{r+1}}}, \ldots, \hat{\nabla}^{Giv}_{z \p_{\chi_{r+e}}})(1) &= 0\, . \label{eq:mirwelldef}
\end{align}
We are using the presentation $1 = \sum_{i=1}^s \tilde{J}_i s_i$ of the section $1$ on $\mbc^*_z \times W_0$. Since the $s_i$ are flat sections the equations above are equivalent to
\begin{align}
\Box^X_{\underline{l}}(\chi_{1},\ldots,\chi_{r+e},z, {z \chi_1 \p_{\chi_1}}, \ldots, {z \chi_r \p_{\chi_r}}, {z \p_{\chi_{r+1}}}, \ldots, {z \p_{\chi_{r+e}}}) ((id_{\mbc_z \times \tau})^* \tilde{J}_i) ((id_{\mbc_z \times \tau})^* \tilde{J}_i)
=\; &\;0\, , \notag \\
\check{E}(\chi_{1},\ldots,\chi_{r+e},z, {z \chi_1 \p_{\chi_1}}, \ldots, {z \chi_r \p_{\chi_r}}, {z \p_{\chi_{r+1}}}, \ldots, {z \p_{\chi_{r+e}}}) ((id_{\mbc_z \times \tau})^* \tilde{J}_i)
=\; &\;0\, . \notag
\end{align}
But this follows from Theorem \ref{thm:mirthm} and Lemma \ref{lem:propI}. Since the equations \eqref{eq:mirwelldef} hold on $\mbc^*_z \times \times W_X$ the hold on $\mbc_z \times W_X$ by continuity.  In order to show that they are isomorphic it is enough to prove this on the germs at $0$ (since we are allowed to shrink $W_X$ if necessary). By Nakayama's lemma it is even enough to show this on the fiber over $0$. But this is clear since both fibers are isomorphic to $H^*_{orb}(\mcx)$ and the action of $\hat{\nabla}^{Giv}_{z \chi \p_{\chi_i}}$ and $\hat{\nabla}^{Giv}_{z \chi_j}$ resp. $z \chi \p_{\chi_i}$ resp. $z \chi_j$ for $i=1,\ldots ,r$  and $j= r+1,\ldots ,r+e$ generate the fibers at $0$. It remains  to show that this isomorphism extends to an isomorphism of $\trTLEPlog$-structures.

Denote by $D \subset W_X$ the normal-crossing divisor given by $\chi_1\cdot \ldots \cdot \chi_r = 0$. We will show that the extensions to $\{z = \infty\} \times \mbp^1_z \setminus \{0\} \times D$ coincide under the isomorphism $\gamma$. First notice that $\gamma$ gives an identification of local systems $((\qmclogo^{log,X})^{an}_{\mid\mbc^*_z \times W_X})^\nabla \simeq  ((id_{\mbc^*_z} \times \tau)^* \mcg^{ss}_{\mid \mbc^*_z \times W_X})^\nabla$. The extension is then encoded by the $\mbz^{r+1}$-filtrations $(F'_\bullet, P_{\bullet})$ resp. $(\tilde{F}'_\bullet, \tilde{P}_{\bullet})$. Since we already know that the extension over $\mbc^*_z \times D$ coincide we conclude that $P_{\bullet} = \tilde{P}_\bullet$. Hence it is enough to show $F'_{\bullet} = \tilde{F}'_\bullet$. Arguing as in Proposition \ref{prop:LogTrTLEP} it is enough to show that the extensions over $\mbp^1_z \times \{ \underline{\chi} = 0\}$ coincide. But this is clearly the case since the subspace $F_X$ which generates the extension of $\mce_X$ is identified  under $\gamma$ with the subspace $\mbc[T_1,\ldots,\ldots , T_{r+e}]$ which generates the extension of $\mcg^{ss}_{\mid \mbc_z \times \{\underline{\chi} = 0\}}$.

\end{proof}

Using the proposition above we can now deduce an isomorphism of logarithmic Frobenius manifolds.

\begin{thm}
There is a unique germ $Mir: (W_X \times \mbc^{s-(r+e)},0 ) \lra (V,0)$ which identifies the logarithmic Frobenius manifold coming from the big orbifold quantum cohomology (cf. Proposition \ref{prop:logFrobA} to the one coming from the Landau-Ginzburg model (cf. Theorem \ref{thm:logfrobB}). Its restriction to $W_X$ corresponds to the isomorphism $\theta$ of $\trTLEPlog$-structures.
\end{thm}
\begin{proof}
This follows from Proposition \ref{prop:isolog} and the uniqueness statement in  Theorem \ref{thm:ReUnfold}.
\end{proof}

\section{Crepant resolutions and global $tt^*$-geometry}

In this section we will first recall the notion of a (pure and polarized) variation of TERP-structures. If a TERP-structure is pure and polarized it gives rise to $tt^*$-geometry on the underlying space. We will show that the quantum $\mcd$-module of a toric orbifold $X$ underlies such an variation pure and polarized TERP-structures. Our main result is that if $X$ admits a crepant resolution $Z$ than the pure and polarized TERP-structures glue which gives global $tt^*$-geometry.

\begin{defn}[\protect{\cite[Definition 2.12]{He4}, \cite[Definition 2.1]{HS1}}]
Let $\mcm$ be a complex manifold and $n \in \mbz$. A variation of $TERP$-structures on $\mcm$ of weight $n$ consists of the following set of data
\begin{enumerate}
\item A holomorphic vector bundle $\mch$ on $\mbc_z \times \mcm$ 
\item A $\mbr$-local system $\mcl$ on $\mbc^*_z \times \mcm$, together with an isomorphism
\[
\mcl \otimes_\mbr \mco^{an}_{\mbc^*_z \times \mcm} \lra \mch^{an}_{\mid \mbc^*_z \times \mcm}
\]
such that the induced connection extends to a meromorphic connection $\nabla$ on $\mch$ such that $\nabla$ has a pole of Poincar\'{e} rank 1 along $\{0\} \times \mcm$.
\item A polarization $P: \mcl \otimes \iota^* \mcl \lra i^n \underline{\mbr}_{\mbc^*_z \times \mcm}$, which is $(-1)^n$ symmetric and which induces a non-degenerate pairing 
\[
P: \mch \otimes_{\mbc_z \times \mcm} \iota^* \mch \lra z^n \mco_{\mbc_z \times \mcm}
\]
where non-degenerate means that the induced symmetric pairing $[z^{-w}P]:\mch/z \mch \otimes \mch / z \mch \ra \mco_\mcm$ is non-degenerate.
\end{enumerate}
\end{defn}
We now state the definition of a pure and polarized $TERP$-structure.
\begin{defn}
Let $(\mch,\mcl,P,n)$ be a variation of $TERP$-structures on $\mcm$. Let $\overline{\mcm}$ be the complex manifold with the conjugate complex structure and $\gamma: \mbp^1 \times \mcm\ra \mbp^1 \times \mcm$ be the involution $(z,x) \mapsto (\overline{z}^{-1},x)$. Consider $\overline{\gamma^* \mch}$ which is a holomorphic vector bundle on $(\mbp^1 \setminus \{ 0\}) \times \overline{\mcm}$ . Let $\mco_{\mbp^1} \mcc^{an}_\mcm$ be the subsheaf of $\mcc^{an}_{\mbp^1 \times \mcm}$ consisting of functions which are annihilated by $\p_{\overline{z}}$. Define a locally free $\mco_{\mbp^1}\mcc^{an}_\mcm$-module $\hat{\mch}$ by glueing $\mch$ and $\overline{\gamma^* \mch}$ via the following identification on $\mbc^*_z \times \mcm$:
Let $x \in \mcm$ and $z \in \mbc^*_z$ and define
\begin{align}
c: \mch_{\mid(z,x)} &\lra (\overline{\gamma^* \mch})_{\mid (z,x)}\, , \notag \\
a\quad  &\mapsto \quad \nabla\text{-parallel transport of } \overline{z^{-n} \cdot a}\, .\notag
\end{align}
Then $c$ is an anti-linear involution and identifies $\mch_{\mid \mbc^*_z \times \mcm}$ with $\overline{\gamma^* \mch}_{\mid \mbc^*_z \times \overline{\mcm}}$. The involution $c$ restricts to complex conjugation (with respect to $\mcl$) in the fibres over $S^1 \times \mcm$.
\begin{enumerate}
\item $(\mch, \mcl,P,n)$ is called pure iff $\hat{\mch} = p^* p_* \hat{\mch}$, where $p: \mbp^1 \times \mcm \ra \mcm$.
\item Let $(\mch,\mcl,P,n)$ be pure, then
\begin{align}
h: p_* \hat{\mch} \otimes_{\mcc^{an}_\mcm} p_* \hat{\mch} &\lra \mcc^{an}_\mcm \notag \\
(s,t) &\mapsto z^{-n} P(s,c(t)) \notag
\end{align}
is a hermetian form on $p_* \hat{\mch}$. We call $(\mch,\mcl,P,n)$ a pure and polarized $TERP$-structure if this form is positive definite.
\end{enumerate}
\end{defn}

\begin{thm}\label{thm:indTERP}
The restriction of the quantum $\mcd$-module $\mcg$ of a toric orbifold  to $\mbc_z \times W_X \setminus D_X$ underlies a variation of pure and polarized $TERP$-structures of weight $n$.
\end{thm}
\begin{proof}
The proof carries over almost word for word from the manifold case in \cite[theorem 5.3]{RS1}. So we just give a sketch of the proof and refer the reader to loc. cit. for details. Using the mirror isomorphism  $\theta: ({_0}\wqma^{log,X})_{\mid \mbp^1_z \times W_X} \ra (id_{\mbp^1_z} \times \tau)^* \mcg^{ss}_{\mid \mbp^1_z \times W_X}$ it is enough to show that $\qmclogo$ underlies a variation of pure and polarized TERP-structures. Notice that the underlying $\mcd$-module ${^\circ}\qma$ is isomorphic to $\FL^{loc}_{\mcm^\circ_\mcx}\mch^0 W_+ \mco_{Y \times \mcm^\circ_\mcx} $ by the description \eqref{eq:LGLaurent} and Proposition \ref{prop:RSDmodds} (2). The Riemann-Hilbert correspondence gives  $DR(\mch^0 W_+ \mco_{Y \times \mcm^\circ_\mcx}) \simeq {^p}\mch^0 RW_* \underline{\mbc}_{Y \times \mcm^\circ_\mcx}$. Therefore $DR(\mch^0 W_+ \mco_{Y \times \mcm^\circ_\mcx})$ carries a real structure  ${^p}\mch^0 RW_* \underline{\mbr}_{Y \times \mcm^\circ_\mcx}$. It follows from \cite[Theorem 2.2]{Sa1} that the local system of flat sections of $({^\circ}\qma,\nabla)$ is equipped with a real structure. That $({^\circ}\qma$ is pure and polarized follows from \cite[Theorem 4.10]{Sa2}.
\end{proof}

The proof of the theorem above shows that variation of pure and polarized TERP-structures exists on a Zariski open subset of the complexified K\"ahler moduli space $\mcm_\mcx$. In the remaining part of the paper we glue the complexified K\"ahler moduli space of a toric orbifold $X$ to the complexified K\"ahler moduli space of a crepant resolution and show that the  corresponding variation of TERP-structures also glue on the common domain of definition. This gives the global $tt^*$ geometry.\\

Let $X$ be a simplical, numerical-effective toric variety with fan $\Sigma_X$ and denote by $\mathbf{\Sigma} = (N,\Sigma_X, \mfa)$ be the canonical stacky fan with corresponding Deligne-Mumford stack $\mcx$. As above we denote by $a_1 ,\ldots, a_{m+e}$ the images of the standard generator $e_i$ under the map $\mfa: \mbz^{m+e} \ra N$. Recall that $a_1,\ldots, a_m$ are the primitive generators of the $1$-dimensional cones of $\Sigma_X$ and that $Gen(\Sigma_X) = \{a_{m+1}, \ldots, a_{m+e}\}$ (cf.Section \ref{subsec:Exstacky}).  Assume that there exists a crepant toric resolution$\pi: Z \ra X$ of $X$. We denote by $\Sigma_Z$ the corresponding fan.  The rays of $\Sigma_Z$ are denoted by $\rho_1, \ldots , \rho_{m}, \rho_{m+1}, \ldots ,\rho_{t}$ and the primitive integral generators by $b_1, \ldots, b_{m},b_{m+1}, \ldots ,b_t$, where we can assume that $a_i = b_i$ for $i=1,\ldots,m$ since $Z$ is a resolution of $X$.
The following lemma is well-known, but the authors could not find a suitable reference.

\begin{lem}
We have the following equivalence
\[
\pi: Z \lra X \quad \text{is crepant} \qquad \Leftrightarrow \qquad b_{m+1}, \ldots , b_s \in \; \partial \textup{conv}(a_1,\ldots,a_m)
\]
where $\textup{conv}(a_1,\ldots, a_m)$ is the convex hull of $a_1,\ldots, a_m$.
\end{lem}
\begin{proof}
The exceptional divisor of $\pi$ is the divisor $\sum_{i=m+1}^s D_i$, where $D_i$ is the torus invariant divisor corresponding to $\rho_i$.  We write
\[
K_Z = \pi^* K_X + \sum_{i=m+1}^{s} d_i D_i\, 
\]
where the $d_i$ are the discrepancies of $\pi$, i.e. $\pi$ is crepant if $d_i=0$ for all $i \in\{m+1, \ldots , s\}$.  Denote by $\psi_{K_X}: N_\mbq  \lra \mbq$ the piece-wise linear function corresponding to the $\mbq$-Cartier divisor $K_X$.  The pullback of $K_X$ along $\pi$ is represented by the same piece-wise linear function, i.e. $\psi_{\pi^* K_X} = \psi_{K_X}$. Now fix some $k \in \{m+1, \ldots , s\}$. Since $\Sigma_X$ is a complete fan there is a unique minimal cone $\sigma(b_k) \in\Sigma_X(k)$ containing $b_k$. Since the fan $\Sigma_X$ is simplicial we can write uniquely
\[
b_k = \sum_{b_i \in \sigma(b_k)} \kappa_{i} a_{i}\, .
\]
We have $\psi_{\pi^*(K_X)}(b_k)= \sum_{a_i \in \sigma(b_l))} \kappa_{i}$, hence for the discrepancy $d_l$ we get
\[
d_k = \sum_{a_i \in G(\sigma(b_k))} \kappa_{i} -1.
\]
This shows $b_k$ lies in the convex hull of $\{a_i : a_i \in \sigma(b_k)\}$ if and only if $d_l = 0$. Since we assumed $X$ to be nef we have $\partial \textup{conv}(a_1,\ldots,a_m) = \bigcup_{\sigma \in \Sigma_X} \textup{conv}\{a_i \mid a_i \in \sigma\}$, which shows the claim.
\end{proof}
The crepantness of $\pi$ puts several restrictions on $X$.
\begin{lem}
Assume that $\pi: Z \ra X$ is a crepant resolution, then 
\begin{enumerate}
\item $X$ is an $SL$-orbifold.
\item $Gen(\Sigma_X) = \{b_{m+1}, \ldots , b_{m+e}\}$.
\end{enumerate}
\end{lem}
\begin{proof}$ $\\[-1em]
\begin{enumerate}
\item Let $c \in N$ be arbitrary and let $\sigma(c)$ be the unique minimal cone of $\Sigma_X$ containing $c$. As above we can write uniquely $c = \sum_{a_j \in \sigma(c)} \kappa_{j} a_{j}$. The claim is equivalent to the fact that $deg(c) := \sum _{a_j \in \sigma(c)} a_{j} \in \mbn$. Since $\Sigma_Z$ is a subdivision of $\Sigma_X$, hence also complete, we can find a unique minimal cone $\sigma'(c)\in \Sigma_Z(k')$ containing $c$ and with $\sigma'(c) \subset \sigma(c)$. Because $\Sigma_Z$ is regular we can uniquely write $c = \sum_{b_i \in \sigma(c') }  \kappa'_{i} b_{i}$ with $\kappa'_i \in \mbn$.  Hence we have
\[
c = \sum_{b_i \in \sigma'(c)} \kappa'_i b_i = \sum_{b_i \in \sigma'(c)} \kappa'_i (\sum_{a_j \in \sigma(b_i)}\kappa_{ji} a_j)\, .
\]
Because the $\sigma(c)$ was chosen to be minimal and because of the lemma above, this gives
\[
deg(c) = \sum_{b_i \in \sigma'(c)} \kappa'_i (\sum_{a_j \in \sigma(b_i)}\kappa_{ji}) = \sum_{b_i \in \sigma'(c)} a'_i  \cdot 1  \in \mbn\, .
\]
\item Let $\sigma$ be a cone of $\Sigma_X$. First notice that the degree $deg(c)$ of an element $c$ is additive inside a fixed cone, i.e. for $c,c' \in \sigma$ we have $deg (c+c') = deg(c) + deg(c')$.  Because of the first point this shows that $\{b_{m+1}, \ldots , b_{m+e}\} \subset Gen(\Sigma_X)$, since their degree is minimal. Now assume that $c \in Gen(\Sigma_X)$. Because $\Sigma_Z$ is a regular fan, there exists a cone $\sigma' \in \Sigma_Z$ such that $c = \sum_{b_i \in \sigma'} \kappa_i b_i$ with $\kappa_i \in \mbn$. Since $deg(c) =1$, there exists is exactly one $i_0$ with $a_{i_0} = 1$. Since $c \in Gen(\Sigma_X)$, we conclude that $c = b_{i_0}$ for some $i_0 \in \{m+1, \ldots ,m+e\}$.
\end{enumerate}
\end{proof}

We get the following statement from $\deg(a_i) =1 $ for $i =m+1,\ldots, m+e$.
\begin{cor}
The orbifold cohomology $H^*_{orb}(X)$ is $H^2_{orb}$-generated.
\end{cor}

The lemma above shows that $a_i =b_i$ for $i=1,\ldots ,m+e$. Consider the sequence
\begin{equation}\label{eq:exse5}
0 \lra \mbl \lra  \mbz^{m+e} \lra N \lra 0\, .
\end{equation}
Since $Z$ is smooth, we get the exact sequence
\[
0 \lra N^\star \lra (\mbz^{m+e})^\star \simeq PL(\Sigma_Z) \lra \mbl^\star \simeq Pic(Z) \lra 0
\]
when we apply $Hom(-,\mbz)$ to sequence \eqref{eq:exse5}.\\

We get the following commutative diagram with exact rows:
\[
\xymatrix{0 \ar[r] & N^* \ar[r] \ar@{=}[d]& PL(\Sigma_X) \ar[r] \ar@{^{(}->}[d]^\Theta & Pic(X) \ar[r] \ar@{^{(}->}[d]^\theta & 0 \\
0 \ar[r] & N^* \ar[r] \ar@{=}[d]& PL(\Sigma_X^e) \ar[r] \ar@{^{(}->}[d] & Pic^e(X) \ar[r] \ar@{^{(}->}[d] & 0 \\
0 \ar[r] & N^* \ar[r] & PL(\Sigma_Z) \ar[r] & Pic(Z) \ar[r] & 0}
\]

The image of the K\"ahler cone $\overline{\mck}_X$ under the embedding $Pic(X) \otimes \mbq \overset{\theta}{\lra} Pic^e(X)\otimes \mbq  \simeq Pic(Z) \otimes \mbq$, which we denote by $\mck_X$ is a face of the K\"ahler cone $\mck_Z$ by \cite[Theorem 2.5]{OP}. We need the following lemma
\begin{lem}
The images of $D_i \in (\mbz^{m+e})^\star \simeq PL(\Sigma_Z)$ for $i=m+1,\ldots, m+e$ do not lie in $\mck_Z$.
\end{lem}
\begin{proof}
The element $D_i$ seen as a piece-wise linear function on $N$ is given by $D_i(b_j) = \delta_{ij}$. We have $b_i = \sum_{a_j \in \sigma(b_i)} \kappa_j a_j$ wher $\sigma(b_i)$ is the minimal cone in $\Sigma_X$ containing $b_i$. Therefore $1 = D_i(b_i) = D_i(\sum_{b_j \in \sigma(b_i)} \kappa_j b_j) > \sum_{b_j \in \sigma(b_i)} \kappa_j D_i(a_j) = 0$ which shows that $D_i \not \in CPL(\Sigma_Z)$ for $i=m+1,\ldots,m+e$. Since we have $N^\star \subset CPL(\Sigma_Z)$ we see that $[D_i] \in \mck_Z$.
\end{proof} 

The lemma above shows that we get two $r+e$ dimensional cones in $Pic(Z) \otimes \mbq$, namely $\mck_Z$ and $\mck_X^e$  which intersect along the face $\mck_X$. Now consider the lattice $Pic^e(X)$ inside $Pic(Z) \otimes \mbq$. We will choose 
two different $\mbz$-bases for $Pic^e(X)$. The first one is $p_1,\ldots , p_{r+e}$ with the property that $p_1, \ldots , p_r$ is a $\mbz$-basis of the image 
of $\theta$ and $[D_{m+i}] = p_{r+i}$ for $i =1,\ldots ,e$.
The second basis $q_1, \ldots, q_{r+e}$ is chosen such that 
\begin{enumerate}
\item $p_i = q_i$ for $i=1,\ldots,r$,
\item $q_i$ lies in $\mck_Z$ for $i=1,\ldots, r+e$.
\end{enumerate}
Denote the cones generated by $p_1,\ldots, p_{r+e}$ resp. $q_1,\ldots,q_{r+e}$ by $C_X$ resp. $C_Z$.
Let $\Sigma_\mcm$ be the fan (with respect to the lattice $Pic^e(X)$) generated by the cones $C_X$ and $C_Z$ together with its faces. The \emph{global K\"ahler moduli space} $\mcm$ is the smooth toric variety corresponding to the fan $\Sigma_\mcm$. This space is covered by two charts $\mcm_\mcx\simeq \mbc^{r+e}$ resp. $\mcm_Z\simeq \mbc^{r+e}$ corresponding to the cones $C_X$ resp. $C_Z$.\\

If we apply the results of Section \ref{sec:LG} to the toric orbifold $X$ resp. toric manifold $Z$ we get the following isomorphisms of variations of (pure and polarized) $TERP$-structures from Proposition \ref{prop:isolog} and Theorem \ref{thm:indTERP}:
\begin{align}
({_0}\qma^{log,Z})_{\mid \mbc_z \times (W_Z\setminus D_Z)} \lra (id_{\mbc_z} \times \tau_Z)^* \mcg^{ss}_{Z\mid \mbc_z \times (W_Z \setminus D_Z)}\, ,\notag \\
({_0}\qma^{log,\mcx})_{\mid \mbc_z \times (W_\mcx\setminus D_\mcx)} \lra (id_{\mbc_z} \times \tau_\mcx)^* \mcg^{ss}_{\mcx\mid \mbc_z \times (W_\mcx \setminus D_\mcx)}\notag
\end{align}
where $W_Z\subset \mcm_Z^{an}$ resp. $W_\mcx \subset \mcm_\mcx^{an}$ are analytic neighborhoods of $p_Z$ resp. $p_\mcx$.

\begin{rem}
Notice that there is a small caveat here. We have choosen the basis $q_1,\ldots, q_{r+e}$ as a $\mbz$-basis of $Pic^e(X)\subset Pic(Z)$. In order to apply the results of Section \ref{sec:LG} and \ref{sec:QC} in the case $\mcx = Z$ we should choose a basis of $Pic(Z)$ instead of a basis which only generates a sublattice (of finite index). Notice that this requirement is actually not need and was only inserted for the ease of exposition.
\end{rem}

\begin{thm}
There exists a variation of TERP-structures ${_0}\qma$ on the global K\"ahler moduli space $\mcm$ and analytic neighborhoods $W_Z, W_\mcx \subset \mcm^{an}$  of the large volume limits $p_Z, p_\mcx$ such that
\begin{align}
({_0}\qma)_{\mid \mbc_z \times (W_Z \setminus D_Z)} \simeq (id_{\mbc_z} \times \tau_Z)^*\mcg^{ss}_{Z\mid \mbc_z \times(W_Z \setminus D_Z)}\, ,\notag \\
({_0}\qma)_{\mid \mbc_z \times (W_\mcx \setminus D_\mcx)} \simeq (id_{\mbc_z} \times \tau_\mcx)^*\mcg^{ss}_{\mcx\mid \mbc_z \times(W_\mcx \setminus D_\mcx)}\, .\notag 
\end{align}
\end{thm}
\begin{proof}
The proof follows from the fact that ${_0}\qma^{Z} =  {_0}\qma^{\mcx}$ since the sequences $\eqref{eq:exse5}$ are equal for the fan $\Sigma_Z$ and the extended stacky fan $\Sigma^e_X$. 
\end{proof}

\bibliographystyle{amsalpha} 
\bibliography{logdegLG}

\def\cprime{$'$}
\providecommand{\bysame}{\leavevmode\hbox to3em{\hrulefill}\thinspace}
\providecommand{\MR}{\relax\ifhmode\unskip\space\fi MR }
\providecommand{\MRhref}[2]{%
  \href{http://www.ams.org/mathscinet-getitem?mr=#1}{#2}
}
\providecommand{\href}[2]{#2}
\begin{thebibliography}{LLQW14}

\bibitem[BCS05]{BCS}
Lev~A. Borisov, Linda Chen, and Gregory~G. Smith, \emph{The orbifold {C}how
  ring of toric {D}eligne-{M}umford stacks}, J. Amer. Math. Soc. \textbf{18}
  (2005), no.~1, 193--215 (electronic). \MR{2114820 (2006a:14091)}

\bibitem[BG08]{MR2411404}
Jim Bryan and Amin Gholampour, \emph{Root systems and the quantum cohomology of
  {$ADE$} resolutions}, Algebra Number Theory \textbf{2} (2008), no.~4,
  369--390. \MR{2411404}

\bibitem[BG09a]{MR2518631}
\bysame, \emph{Hurwitz-{H}odge integrals, the {$E_6$} and {$D_4$} root systems,
  and the crepant resolution conjecture}, Adv. Math. \textbf{221} (2009),
  no.~4, 1047--1068. \MR{2518631}

\bibitem[BG09b]{MR2551767}
\bysame, \emph{The quantum {M}c{K}ay correspondence for polyhedral
  singularities}, Invent. Math. \textbf{178} (2009), no.~3, 655--681.
  \MR{2551767}

\bibitem[BG09c]{MR2483931}
Jim Bryan and Tom Graber, \emph{The crepant resolution conjecture}, Algebraic
  geometry---{S}eattle 2005. {P}art 1, Proc. Sympos. Pure Math., vol.~80, Amer.
  Math. Soc., Providence, RI, 2009, pp.~23--42. \MR{2483931}

\bibitem[BGP08]{MR2357679}
Jim Bryan, Tom Graber, and Rahul Pandharipande, \emph{The orbifold quantum
  cohomology of {$\Bbb C^2/Z_3$} and {H}urwitz-{H}odge integrals}, J. Algebraic
  Geom. \textbf{17} (2008), no.~1, 1--28. \MR{2357679}

\bibitem[BMP09]{MR2541935}
S.~Boissi{\`e}re, E.~Mann, and F.~Perroni, \emph{The cohomological crepant
  resolution conjecture for {$\Bbb P(1,3,4,4)$}}, Internat. J. Math.
  \textbf{20} (2009), no.~6, 791--801. \MR{2541935}

\bibitem[BMP11]{MR2772168}
Samuel Boissi{\`e}re, {\'E}tienne Mann, and Fabio Perroni, \emph{Computing
  certain {G}romov-{W}itten invariants of the crepant resolution of {$\Bbb
  P(1,3,4,4)$}}, Nagoya Math. J. \textbf{201} (2011), 1--22. \MR{2772168}

\bibitem[CCIT15]{CCIT}
Tom Coates, Alessio Corti, Hiroshi Iritani, and Hsian-Hua Tseng, \emph{A mirror
  theorem for toric stacks}, Compos. Math. \textbf{151} (2015), no.~10,
  1878--1912. \MR{3414388}

\bibitem[CIT09]{MR2529944}
Tom Coates, Hiroshi Iritani, and Hsian-Hua Tseng, \emph{Wall-crossings in toric
  {G}romov-{W}itten theory. {I}. {C}repant examples}, Geom. Topol. \textbf{13}
  (2009), no.~5, 2675--2744. \MR{2529944}

\bibitem[CLLZ14]{MR3202006}
BoHui Chen, AnMin Li, XiaoBin Li, and GuoSong Zhao, \emph{Ruan's conjecture on
  singular symplectic flops of mixed type}, Sci. China Math. \textbf{57}
  (2014), no.~6, 1121--1148. \MR{3202006}

\bibitem[CLZZ09]{MR2553561}
Bohui Chen, An-Min Li, Qi~Zhang, and Guosong Zhao, \emph{Singular symplectic
  flops and {R}uan cohomology}, Topology \textbf{48} (2009), no.~1, 1--22.
  \MR{2553561}

\bibitem[Coa09]{MR2486673}
Tom Coates, \emph{On the crepant resolution conjecture in the local case},
  Comm. Math. Phys. \textbf{287} (2009), no.~3, 1071--1108. \MR{2486673}

\bibitem[CR13]{MR3112518}
Tom Coates and Yongbin Ruan, \emph{Quantum cohomology and crepant resolutions:
  a conjecture}, Ann. Inst. Fourier (Grenoble) \textbf{63} (2013), no.~2,
  431--478. \MR{3112518}

\bibitem[DM13]{Douai-Mann}
Antoine Douai and Etienne Mann, \emph{The small quantum cohomology of a
  weighted projective space, a mirror {$D$}-module and their classical limits},
  Geom. Dedicata \textbf{164} (2013), 187--226. \MR{3054624}

\bibitem[Her02]{He3}
Claus Hertling, \emph{Frobenius manifolds and moduli spaces for singularities},
  Cambridge Tracts in Mathematics, vol. 151, Cambridge University Press,
  Cambridge, 2002.

\bibitem[Her03]{He4}
\bysame, \emph{{$tt\sp *$} geometry, {F}robenius manifolds, their connections,
  and the construction for singularities}, J. Reine Angew. Math. \textbf{555}
  (2003), 77--161.

\bibitem[HS10]{HS1}
Claus Hertling and Christian Sevenheck, \emph{Limits of families of {B}rieskorn
  lattices and compactified classifying spaces}, Adv. Math. \textbf{223}
  (2010), no.~4, 1155--1224. \MR{2581368}

\bibitem[Iri09]{Ir2}
Hiroshi Iritani, \emph{{A}n integral structure in quantum cohomology and mirror
  symmetry for toric orbifolds}, Adv. Math. \textbf{222} (2009), no.~3,
  1016--1079.

\bibitem[{Iwa}08]{2008PhDT.........4I}
Y.~{Iwao}, \emph{{Invariance of Gromov-Witten theory under a simple flop}},
  Ph.D. thesis, The University of Utah, 2008.

\bibitem[Jia08]{Jext}
Yunfeng Jiang, \emph{The orbifold cohomology ring of simplicial toric stack
  bundles}, Illinois J. Math. \textbf{52} (2008), no.~2, 493--514. \MR{2524648
  (2011c:14059)}

\bibitem[JT08]{JT}
Yunfeng Jiang and Hsian-Hua Tseng, \emph{Note on orbifold {C}how ring of
  semi-projective toric {D}eligne-{M}umford stacks}, Comm. Anal. Geom.
  \textbf{16} (2008), no.~1, 231--250. \MR{2411474 (2009h:14092)}

\bibitem[LLQW14]{2014arXiv1401.7097L}
Y.-P. {Lee}, H.-W. {Lin}, F.~{Qu}, and C.-L. {Wang}, \emph{{Invariance of
  quantum rings under ordinary flops: III}}, ArXiv e-prints (2014).

\bibitem[LLW11]{2011arXiv1109.5540L}
Y.-P. {Lee}, H.-W. {Lin}, and C.-L. {Wang}, \emph{{Invariance of Quantum Rings
  under Ordinary Flops I: Quantum corrections and reduction to local models}},
  ArXiv e-prints (2011).

\bibitem[LLW13]{2013arXiv1311.5725L}
\bysame, \emph{{Invariance of Quantum Rings under Ordinary Flops II: A quantum
  Leray--Hirsch theorem}}, ArXiv e-prints (2013).

\bibitem[OP91]{OP}
Tadao Oda and Hye~Sook Park, \emph{Linear gale transforms and
  {G}elfand-{K}apranov-{Z}elevinskij decompositions}, Tohoku Math. J. (2)
  \textbf{43} (1991), no.~3, 375--399.

\bibitem[Per07]{MR2360646}
Fabio Perroni, \emph{Chen-{R}uan cohomology of {$ADE$} singularities},
  Internat. J. Math. \textbf{18} (2007), no.~9, 1009--1059. \MR{2360646}

\bibitem[Pha79]{Ph1}
Fr{\'e}d{\'e}ric Pham, \emph{Singularit\'es des syst\`emes diff\'erentiels de
  {G}auss-{M}anin}, Progress in Mathematics, vol.~2, Birkh\"auser Boston,
  Mass., 1979, With contributions by Lo Kam Chan, Philippe Maisonobe and
  Jean-{\'E}tienne Rombaldi.

\bibitem[Rei09]{Re1}
Thomas Reichelt, \emph{A construction of {F}robenius manifolds with logarithmic
  poles and applications}, Comm. Math. Phys. \textbf{287} (2009), no.~3,
  1145--1187.

\bibitem[RS10]{RS1}
Thomas Reichelt and Christian Sevenheck, \emph{Logarithmic {F}robenius
  manifolds, hypergeometric systems and quantum $\mathcal{D}$-modules},
  Preprint math.AG/1010.2118, to appear in ``Journal of Algebraic Geometry'',
  2010.

\bibitem[RS12]{RS2}
T.~{Reichelt} and C.~{Sevenheck}, \emph{{Non-affine Landau-Ginzburg models and
  intersection cohomology}}, ArXiv e-prints (2012).

\bibitem[Rua06]{MR2234886}
Yongbin Ruan, \emph{The cohomology ring of crepant resolutions of orbifolds},
  Gromov-{W}itten theory of spin curves and orbifolds, Contemp. Math., vol.
  403, Amer. Math. Soc., Providence, RI, 2006, pp.~117--126. \MR{2234886}

\bibitem[Sab97]{Sa1}
Claude Sabbah, \emph{Monodromy at infinity and {F}ourier transform}, Publ. Res.
  Inst. Math. Sci. \textbf{33} (1997), no.~4, 643--685.

\bibitem[Sab08]{Sa2}
\bysame, \emph{{F}ourier-{L}aplace transform of a variation of polarized
  complex {H}odge structure.}, J. Reine Angew. Math. \textbf{621} (2008),
  123--158.

\bibitem[TW12]{TW}
H.-H. {Tseng} and D.~{Wang}, \emph{{Seidel Representations and Quantum
  Cohomology of Toric Orbifolds}}, ArXiv e-prints (2012).

\end{thebibliography}

\end{document}